\documentclass[a4paper,11pt,twoside]{amsart}
\usepackage[english]{babel}
\usepackage[utf8]{inputenc}
\usepackage{mathabx}
\usepackage[a4paper,inner=3cm,outer=3cm,top=4cm,bottom=4cm,pdftex]{geometry}
\usepackage{fancyhdr}
\usepackage{enumerate}
\usepackage{color}
\usepackage{bm}
\usepackage{bold-extra}

\usepackage{color}
\usepackage{bold-extra}
\usepackage{ mathrsfs }

\usepackage{comment}
\usepackage{graphics}
\usepackage{aliascnt}
\usepackage[pdftex,citecolor=green,linkcolor=red]{hyperref}

\usepackage{amsmath}
\usepackage{amsfonts}
\usepackage{amssymb}
\usepackage{amsthm}
\usepackage{comment}
\usepackage{mathtools}

\newtheorem{theorem}{Theorem}[section]

\newtheorem{lemma}[theorem]{Lemma}
\newtheorem{proposition}[theorem]{Proposition}
\theoremstyle{definition}
\newtheorem{definition}[theorem]{Definition}
\newtheorem{remark}[theorem]{Remark}

\numberwithin{equation}{section}

\newcommand\eps{\varepsilon}

\newcommand\E{\mathbb{E}}
\newcommand\R{\mathbb{R}}
\newcommand\Z{\mathbb{Z}}

\newcommand\Lv{{350}}
\newcommand\Hv{{39/40}}
\newcommand{\Mod}[1]{\ (\mathrm{mod}\ #1)}
\renewcommand\d{\, \textnormal{d}} 

\title{Primes in arithmetic progressions and short intervals without $L$-functions}

\author{Kaisa Matom\"aki}
\address{Department of Mathematics and Statistics, University of Turku, 20014 Turku, Finland}
\email{ksmato@utu.fi}

\author{Jori Merikoski}
\address{Mathematical Institute, University of Oxford\\ Radcliffe Observatory Quarter, Woodstock Rd\\
Oxford OX2 6GG\\
UK}
\email{jori.merikoski@maths.ox.ac.uk}

\author{Joni Ter\"{a}v\"{a}inen}
\address{Department of Mathematics and Statistics, University of Turku, 20014 Turku, Finland}
\email{joni.p.teravainen@gmail.com}

\begin{document}

\begin{abstract}
We develop a sieve that can  detect primes in multiplicatively structured sets under certain conditions. We apply it to obtain a new $L$-function free proof of Linnik's problem of bounding the least prime $p$ such that $p\equiv a\Mod q$ (with the bound $p \ll q^\Lv$) as well as a new $L$-function free proof that the interval $(x-x^\Hv, x]$ contains primes for every large $x$. In a future work we will develop the sieve further and provide more applications.
\end{abstract}

\maketitle

\section{Introduction}

We develop a sieve that can detect primes in sets that are multiplicative structured in a certain sense. Before stating our sieve results, we discuss applications toward Linnik's problem and primes in short intervals.

\subsection{Applications}

\subsubsection{Linnik's theorem}
Linnik's celebrated theorem \cite{Linnik1, Linnik2} states that there exists an absolute constant $L > 1$ such that if $q\in \mathbb{N}$ and $a$ is coprime to $q$, the least prime $p\equiv a\pmod q$ is always $\ll q^{L}$. The best known exponent here is $L=5$, due to Xylouris \cite{Xy}, refining the work of Heath-Brown~\cite{H-B-Linnik}. Most proofs of Linnik's theorem, including these, heavily use information concerning zeros of Dirichlet $L$-functions.

As an application of our sieve, we obtain a new proof of Linnik's theorem which is ``$L$-function free'' in the sense that the zeros of Dirichlet $L$-functions do not play any role in the proof.

\begin{theorem}\label{thm_Linnik}
Let $q\geq q_0$, where $q_0$ is an effectively computable large integer. Then, for every $a \in \mathbb{Z}_q^\times$, there exists a prime $p \leq q^{\Lv}$ such that $p \equiv a \Mod{q}$.
\end{theorem}

We have not optimized the exponent (see Remarks~\ref{rem:SecondLargest},~\ref{rem:UseMoreSn},~\ref{rem:Burgess},~\ref{rem:eta+-impr}, and~\ref{rem:eta+impr} for more discussion) but rather tried to use as simple tools as possible. Previously, Friedlander and Iwaniec gave a sieve-theoretic proof of Linnik's theorem in~\cite[Chapter 24]{Opera} and made it explicit in a recent series of two papers \cite{FI1, FI2}. Their proof produced the value $L=75\, 744\, 000$, and it relied on the classical zero-free region. Granville, Harper and Soundararajan \cite{GHS} gave an $L$-function-free proof of Linnik's theorem using  pretentious number theory. They do not give an explicit value for $L$, but mention (private communication) that it would also be rather large. A major advantage of our sieve is that the numerical values are not only better but also very easy to compute in practice, in comparison to other methods. 

\subsubsection{Primes in short intervals}
Hoheisel \cite{Hoheisel} proved that there is some exponent $\alpha<1$ such that for all sufficiently large $x$ there exists a prime number in the interval $(x-x^\alpha, x]$. The current records for the number of primes in a short interval $(x-x^\alpha, x]$ are an asymptotic formula for $\alpha> 7/12$ by Huxley \cite{Huxley} (improved to $\alpha=7/12-o(1)$ by Heath-Brown \cite{HBprimesinshort}) and a correct order lower bound for $\alpha =0.525$ by Baker, Harman and Pintz \cite{Huxley}.  All of the above-mentioned papers rely heavily on the theory of the $\zeta$-function, especially on a strong zero-free region. 

As an application of our sieve we get an elementary~\footnote{By elementary we mean that the proof does not use complex analysis. It uses linear sieve which has (e.g. in~\cite[Chapter 11]{Opera}) an elementary proof, avoiding complex analysis and prime number theorem} proof of the following.
\begin{theorem}
\label{thm:Primesinshorts}
Let $\alpha = \Hv$ and let $x \geq 2$ be sufficiently large. Then 
\[
\sum_{x-x^\alpha < p \leq x} 1 \geq \frac{1}{40} \cdot \frac{x^\alpha}{\log x}
\]
\end{theorem}
Granville, Harper, and Soundararajan \cite{GHS} gave a pretentious proof of Hoheisel's theorem, which approaches an asymptotic formula as $\alpha \to 1$. It is not clear for what value of $\alpha < 1$ they could obtain a positive lower bound, but such $\alpha$ is likely to be very close to $1$.

The proof of Theorem~\ref{thm:Primesinshorts} is easier than the proof of Theorem~\ref{thm_Linnik}. The deduction from the relevant sieve theorem is quickly done in Section~\ref{ssec:PrimesInShorts} while the proof of the short interval sieve theorem is finished in Section~\ref{sec:shortspecialproof}.

\subsection{Sieve theorems}
Before stating our sieve theorem, we need to introduce some notation. Throughout, $\mathcal{N}$ will denote a multiplicatively closed subset of $\mathbb{N}$ (i.e. a subset such that $a, b \in \mathcal{N} \implies ab \in \mathcal{N}$). For any real number $x \geq 1$ and any set $\mathcal{B} \subset \mathbb{N}$, we write 
\[
\mathcal{B}(x) = \{n \in \mathcal{B}, n \leq x\}.
\]
\begin{definition}[Level of distribution]
\label{def:lod}
Let $\mathcal{N} \subseteq \mathbb{N}$ be multiplicatively closed, $x, X \geq 2$, and $k, \theta \in \mathbb{R}_+$. Let $(a_n)_{n\in \mathcal{N}(x)}$ be a sequence of complex numbers, and let $h\colon \mathcal{N}\to \mathbb{R}_{\geq 0}$ be a multiplicative function such that $h(p) < p$ for every prime $p \in \mathcal{N}$. 

We say that the sequence $(a_n)_{n\in \mathcal{N}(x)}$ has \emph{level of distribution $\theta$ with dimension $k$, size $X$, and density $h$} if
\begin{align*}
\sum_{\substack{d \in \mathcal{N}(x^{\theta})}} \mu^2(d) \left|\sum_{\substack{n \leq x/d \\ dn \in \mathcal{N}}} a_{dn} - \frac{h(d)}{d} X \right|  = o \bigg(X \prod_{p < x} \bigg( 1- \frac{h(p)}{p} \bigg) \bigg),
\end{align*}
and if there exists a constant $C\geq 1$ such that for all $z\geq w\geq 2$ we have
\begin{align}
\label{eq:h(p)condition}
\prod_{\substack{w\leq p< z \\ p \in \mathcal{N}}}\left(1-\frac{h(p)}{p}\right)^{-1}\leq \left(1+\frac{C}{\log w}\right)\left(\frac{\log z}{\log w}\right)^k.
\end{align}
\end{definition}

Next we state two different sieve theorems, which will be used in the proofs of Theorems~\ref{thm_Linnik} and~\ref{thm:Primesinshorts}, respectively. As we will discuss in Section~\ref{ssec:future}, in a forthcoming work we will develop more general versions of the sieve theorems and give further applications.

\subsubsection{The sieve theorem for finite groups}

For $\delta > \eta > 0$, we define 
\begin{align}
\label{eq:calSpecFdef}
\mathcal{F}(\delta,\eta)\coloneqq 1-\frac{1-\delta}{\delta^2}\max\left\{\frac{1}{4}\sqrt{1+(4\delta-1)^2},\delta-\eta\right\}.
\end{align}
Here we state a sieve theorem that suffices for our application to Linnik's problem. We will prove this in Section~\ref{sec:non-torsion-specialproof}.
\begin{theorem} 
 \label{thm:non-torsion-special}
Let $\eta \in (0, 1)$ and $\kappa\in (2, 2+1/12)$, and let $\theta \colon [1/2, 1] \to (1/2, 1)$ be a fixed increasing function. Let $\mathcal{N} \subseteq \mathbb{N}$ be multiplicatively closed, let $G$ be a finite (multiplicative) abelian group, and let $a \in G$. Let $f\colon \mathcal{N}\to G$ and $h \colon \mathcal{N} \to \mathbb{R}_{\geq 0}$ be multiplicative. Suppose that the following hold for some $Y > 0$ and $x \geq 2$ (where $o(1) \to 0$ when $x \to \infty$).

\begin{enumerate}[({A}1)]
\item \label{hyp:lod-special}(Level of distribution) For any $y = x^{\rho} \in [x^{1/2},x]$, the sequence $(1_{f(n)=a})_{n\in \mathcal{N}(y)}$ has level of distribution $\theta(\rho)$ with dimension $1$, size $y/Y$ and density $h$.
\item \label{hyp:BT-special}(Brun--Titchmarsh estimate) For each $g\in G$, we have 
\[
\sum_{\substack{x^{1/6} < p \leq x^{1/3} \\ p \in \mathcal{N}}} \frac{1_{f(p)=g}}{p} \leq  \frac{\kappa}{|G|} \sum_{\substack{x^{1/6} < p \leq x^{1/3} \\ p \in \mathcal{N}}} \frac{1}{p} 
\]
\item \label{hyp:MassLong-special} ($\mathcal{N}$ contains most primes) We have
\[
\sum_{\substack{x^{1/6} < p \leq x^{1/3} \\ p \in \mathcal{N}}} \frac{1}{p} = \log 2 + o(1).
\]
\item \label{hyp:index2-special}(Mass in index $2$ cosets) For any subgroup $H\leq G$ of index $2$, we have
\[
\sum_{\substack{x^{1/6} < p \leq x^{1/3} \\ p \in \mathcal{N}}} \frac{1_{f(p)\in a H}}{p} \geq 
 \eta \cdot \frac{1}{2}\sum_{\substack{x^{1/6} < p \leq x^{1/3} \\ p \in \mathcal{N}}} \frac{1}{p} 
\]
\end{enumerate}
Then
\begin{align}
\label{eq:FinalLowSpec}
\begin{aligned}
&\sum_{\substack{x^{1/2} < p \leq x \\ p \in \mathcal{N}}} \frac{1_{f(p)=a}}{p} \geq \frac{1}{|G|}\left(\frac{\log^3 2}{3} \mathcal{F}\left(\frac{1}{\kappa},\frac{\eta}{\kappa}\right) + o(1)\right) \\
& \, - \frac{e^{\gamma} \log x}{Y} \int_{1/2}^1 \int_{\theta(\rho)/2}^{1/2} \prod_{\substack{p < x^{\frac{\rho}{2} (\theta(\rho)-\beta)} \\ p \in \mathcal{N}}} \left(1-\frac{h(p)}{p}\right)\frac{\d\beta}{\beta} \d\rho + o\left(\frac{\log x}{Y} \prod_{\substack{p \leq x \\ p \in \mathcal{N}}} \left(1-\frac{h(p)}{p}\right)\right).
\end{aligned}
\end{align}
\end{theorem}

\begin{remark}
The assumption~(A4) ensures that we do not claim to overcome the parity barrier, and the theorem would be false without an assumption of this form. If one e.g. chooses $\mathcal{N} = \mathbb{N}$, $G = \{-1, 1\},$ $f(n) = \lambda(n)$, and $a = 1$, then the left-hand side of~\eqref{eq:FinalLowSpec} vanishes but one can show that the condition~(A1) holds for $\theta(\rho) = 1 - \varepsilon$ for any $\varepsilon > 0$, that the condition~(A2) holds with $\kappa = 2+\varepsilon$ for any $\varepsilon > 0$, and (A3) holds. Once $\varepsilon$ is sufficiently small, the right hand side of~\eqref{eq:FinalLowSpec} is positive. However, the condition~(A4) of course does not hold. 
\end{remark}

\begin{remark} 
\label{rem:SpecSimpleMT}
Notice that the result is nontrivial only if the problem is genuinely one-dimensional in the sense that
\[
\prod_{\substack{p \leq x \\ p \in \mathcal{N}}} \left(1-\frac{h(p)}{p}\right) \ll \frac{Y}{|G|} \cdot \frac{1}{\log x}.
\]
Actually, in typical applications we have
\[
\prod_{\substack{p \leq y \\ p \in \mathcal{N}}} \left(1-\frac{h(p)}{p}\right) = (1+o(1))\frac{Y}{|G|} \cdot \frac{e^{-\gamma}}{\log y}
\]
whenever $y \geq x^{\varepsilon}$ and $\varepsilon > 0$,
and in this case~\eqref{eq:FinalLowSpec} simplifies to
\begin{align}
\label{eq:FinalLowSpecSimple}
\begin{aligned}
\sum_{\substack{x^{1/2} < p \leq x \\ p \in \mathcal{N}}} \frac{1_{f(p)=a}}{p} &\geq \frac{1}{|G|}\left(\frac{\log^3 2}{3} \mathcal{F}\left(\frac{1}{\kappa},\frac{\eta}{\kappa}\right) - \int_{1/2}^1 \int_{\theta(\rho)/2}^{1/2} \frac{2}{\theta(\rho)-\beta} \frac{\d\beta}{\beta} \frac{\d\rho}{\rho} + o(1) \right).
\end{aligned}
\end{align}
Note that then the constant $C$ in \eqref{eq:h(p)condition} is of size $C \gg |G|/Y$. Since $x$ is assumed to be sufficiently large in terms of $C$ in \eqref{eq:h(p)condition}, it is implicit in the hypotheses that $x$ is sufficiently large in terms of $|G|/Y$. While this imposes no restriction e.g. in the case of Linnik's problem (where we take $|G|/Y=\varphi(q)/q< 1$), this restriction on $x$ in terms of $|G|/Y$ could be relaxed by dealing with prime factors $p < x^{o(1)}$ separately by the fundamental lemma of the sieve, or by increasing the sieve dimension from $1$ to $1+o(1)$.
\end{remark}

\begin{remark}
The limitation $\kappa \in (2, 2+1/12)$ comes from Lemma~\ref{le:Smm=4} below and could be relaxed to $\kappa > 2$ if, in the definition of $\mathcal{F}(\delta, \eta)$ in~\eqref{eq:calSpecFdef}, we replaced $\frac{1}{4}\sqrt{1+(4\delta-1)^2}$ by $\max_{m \geq 3} S(\delta, m)$ with $S(\delta, m)$ as in~\eqref{eq:Sdelmdef} below. It would definitely be possible to obtain a better lower bound in case $\kappa \leq 2$. However, we always have $\kappa \geq 2$ when only type I information is available; for instance, for Linnik's problem~(A2) is deduced from the classical Brun--Titchmarsh bound
\begin{align*}
\sum_{\substack{p \leq y \\ p \equiv a \, (q)}} 1 \leq\frac{(2 + o(1)) x }{\varphi(q)\log(y/q)}.
\end{align*}

The condition (A3) could be relaxed by a constant, but this would affect the final lower bound~\eqref{eq:FinalLowSpec}.
\end{remark}

\subsubsection{The sieve theorem for short intervals} 

For $\delta_1, \delta_2, \delta_3 \in (0,1]$, we define 
\begin{align}
\label{eq:calShortsFdef}
\mathcal{F}^\sharp(\delta_1, \delta_2, \delta_3)\coloneqq  1 -\frac{|\sin (\pi \delta_2)|}{\pi \delta_1 \delta_2 \delta_3} \sqrt{(\delta_1-\delta_1^2)(\delta_3-\delta_3^2)}.
\end{align}
Theorem~\ref{thm:Primesinshorts} is a corollary of the following sieve theorem, which will be proved in Section~\ref{sec:shortspecialproof}. The deduction of Theorem~\ref{thm:Primesinshorts} from Theorem~\ref{thm:shorts-special} is easy and will be done in Section~\ref{ssec:PrimesInShorts}.

\begin{theorem} 
 \label{thm:shorts-special}
Let $1/2 < \theta < \alpha <1$, let $\kappa \colon [(\theta-1/2)/3, 1/2] \to (2, \infty)$ be a continuous decreasing function, let $Q \in \mathbb{N}$ and $K = Q^4!$ be such that $K = x^{o(1)}$. Let $f\colon \mathbb{N}\to \mathbb{R}_{\geq 0}$ and $h \colon \mathbb{N} \to \mathbb{R}_{\geq 0}$ be multiplicative. Let $x \geq 2$ and $Y >0$. Suppose that the following hold (where $o(1) \to 0$ when $x \to \infty$).

\begin{enumerate}[({A}1*)]
\item \label{hyp:lod-shorts} (Level of distribution) The sequence $(f(n) \mathbf{1}_{x-x^\alpha < n \leq x})_{n \leq x}$ has level of distribution $\theta$ with dimension $1$, size $x^\alpha /Y$ and density $h$.
\item \label{hyp:BT-shorts}(Brun--Titchmarsh estimate) Let $\Delta \coloneqq x^{\alpha-1}/(K+3)$. For any $y = x^\rho \in [x^{(\theta-1/2)/3}, x^{1/2}]$, we have 
\[
\sum_{\substack{(1-\Delta)y < p \leq y}} f(p) \leq \kappa(\rho) \frac{\Delta y}{\log y}.
\]
\item \label{hyp:long-ints}(Mass on long intervals) Whenever $\frac{\theta-1/2}{3} \leq \beta_1 \leq \beta_2 \leq \frac{1}{2}$ are such that $\beta_2 - \beta_1 \geq \frac{1}{2K^3}$, we have 
\[
\sum_{\substack{x^{\beta_1} < p \leq x^{\beta_2}}} \frac{f(p) \log p}{p} \geq (\beta_2-\beta_1+o(1)) \log x.
\]
\end{enumerate}
Then
\begin{align}
\label{eq:FinalLowShorts}
\begin{aligned}
&\sum_{\substack{x-x^\alpha < p \leq x}} f(p) \geq \frac{x^{\alpha}}{\log x} \int_{(\theta-1/2)/3}^{1/3} \int_{\max\{\alpha_2, \theta-3\alpha_2\}}^{\min\{1/2, 1-2\alpha_2\}} \frac{\mathcal{F}^\sharp(\frac{1}{\kappa(\alpha_1)}, \frac{1}{\kappa(\alpha_2)}, \frac{1}{\kappa(1-\alpha_1-\alpha_2)})}{\alpha_1 \alpha_2 (1-\alpha_1 - \alpha_2)} \d\alpha_1 \d\alpha_2 \\
&- \frac{e^{\gamma} x^\alpha}{Y} \int_{\theta/2}^{1/2} \prod_{\substack{p < x^{(\theta-\beta)/2}}} \left(1-\frac{h(p)}{p}\right)\frac{\d\beta}{\beta} + o\left(\frac{x^\alpha}{Y} \prod_{p < x} \left(1- \frac{h(p)}{p}\right)\right) + o_{Q \to \infty}\left(\frac{x^\alpha}{\log x}\right).
\end{aligned}
\end{align}
\end{theorem}

\begin{remark} 
\label{rem:ShortsSimpleMT}
Similarly to Theorem~\ref{thm:non-torsion-special}, the result is nontrivial only if the problem is genuinely one-dimensional.  But now in typical applications we have
\[
\prod_{\substack{p \leq y}} \left(1-\frac{h(p)}{p}\right) = (1+o(1)) Y \frac{e^{-\gamma}}{\log y}
\]
whenever $y \geq x^\varepsilon$ and $\varepsilon > 0$, so that
\begin{align*}
\frac{e^\gamma x^\alpha}{Y} \int_{\theta/2}^{1/2} \prod_{\substack{p < x^{(\theta-\beta)/2}}} \left(1-\frac{h(p)}{p}\right)\frac{\d\beta}{\beta} = \frac{x^\alpha}{\log x} \left(\frac{2}{\theta}\log \frac{1}{2\theta-1} + o(1)\right).
\end{align*}
\end{remark}
\begin{remark}
The assumption (A3*) could be relaxed by a multiplicative constant, but this would affect the final lower bound~\eqref{eq:FinalLowShorts}.
\end{remark}
\begin{remark}
Theorem~\ref{thm:shorts-special} can be viewed as a sort of a non-compact generalization of Theorem~\ref{thm:non-torsion-special}, with the group being $(\R_{>0},\times)$ and the condition $f(p)=a$ replaced by $x-x^\alpha < p \leq x$. Then the analogue of the assumption (A4) of Theorem \ref{thm:non-torsion-special} is immediate since there are no subgroups of index 2 of $(\R_{>0},\times)$, which explains the apparent difference in the two results.
\end{remark}

\section{Discussion of the proofs}
\subsection{Proof of Theorem~\ref{thm:non-torsion-special}}
The proof of Theorem~\ref{thm:non-torsion-special} is based on combining a linear sieve identity with additive combinatorics. In the first step of the proof we use the linear sieve in a precise form to reduce the problem of lower bounding
\[
\sum_{\substack{x^{1/2} < p \leq x \\ p \in \mathcal{N}}} \frac{\mathbf{1}_{f(p) = a}}{p}
\] 
to the problem of finding a sufficiently good lower bound for the logarithmic number of  products of three primes
\[
\sum_{\substack{x^{1/6} < p_1, p_2, p_3 \leq x^{1/3} \\ p_1, p_2, p_3 \in \mathcal{N}}} \frac{\mathbf{1}_{f(p_1 p_2 p_3) = a}}{p_1 p_2 p_3}.
\] 
We note that a qualitative version of a statement like this follows from Bombieri's sieve which yields that, if the level of distribution function is $\theta(\rho) = 1-\varepsilon$ for a sufficiently small $\varepsilon$, then in order to find primes it suffices to find products of three primes (see e.g.~\cite[Theorem 16.1]{Opera}). 

Friedlander and Iwaniec used a similar idea in their proof of Linnik's theorem in~\cite[Chapter 24]{Opera}, reducing Linnik's problem to the study of products of five primes. Our main new innovation is introducing a new way to find a lower bound for the number of products of a fixed number of primes.

So far we have only used type I information (the level of distribution assumption (A1)). The crucial idea now is that we can use Fourier analysis to get a handle on products of three primes, making use of the trilinear structure. Defining $g\colon G \to \mathbb{R}_{\geq 0}$ by
\[
g(b) \coloneqq \frac{|G|}{\log 2} \cdot \sum_{\substack{x^{1/6} < p \leq x^{1/3} \\ p \in \mathcal{N}}} \frac{1_{f(p)=b}}{p},
\]
it suffices to find a good enough lower bound for the triple convolution
\begin{equation}
\label{eq:tripleConv}
(g \ast g \ast g)(a) \coloneqq \sum_{\substack{b_1, b_2, b_3 \in G\\ b_1 b_2 b_3 = a}} g(b_1) g(b_2) g(b_3).
\end{equation}
By the definition of $g$ and assumptions (A2)--(A4), we see that $g$ essentially takes values in $[0,\kappa]$, has mean value $1$ and, for every index $2$ subgroup $H\leq G$, has mean value at least $\eta$ in the coset $aH$. Normalizing by $1/\kappa$, we can morally think of $g$ as a normalized indicator function of a subset of $G$ with density $1/\kappa < 1/2$.

We will prove a lower bound for~\eqref{eq:tripleConv} using Fourier analysis in $G$, establishing in particular general bounds for the maximum size of complex Fourier coefficients of real-valued weight functions. We note that it would also be possible to use a popular version of Kneser's theorem due to Grynkiewicz \cite{gry} to obtain such a result, but the Fourier argument seems to produce numerically better results while also being more elementary.

Note that in contrast to Harman's prime-detecting sieve method~\cite{HarmanBook}, we only use lower bounds for (trilinear) type II sums rather than an asymptotic formula. On the other hand, we utilize type I information at various scales.

\subsection{Proof of Theorem~\ref{thm:shorts-special}}
Similarly as with Theorem~\ref{thm:non-torsion-special}, a precise form of the linear sieve allows us to reduce the task to proving a lower bound for a sum of products of three primes
\begin{align*}
\sum_{\substack{x^{1/6} < p_1, p_2, p_3 \leq x^{1/2} \\ x-x^\alpha < p_1p_2 p_3 \leq x}} f(p_1)f(p_2)f(p_3).
\end{align*}
After splitting each prime into long intervals (with $K_1 = K_3 = K$ and $K_2 = KQ$)
\begin{align*}
p_j \in (x^{\frac{r_j-1}{2K_j}},x^{\frac{r_j+1}{2K_j}}], \quad \frac{r_1}{2K_1} + \frac{r_2}{2K_2} +\frac{r_3}{2K_3} =1, 
\end{align*}
we can introduce logarithmic factors $\log p_j$ and the task is reduced to obtaining a correct order of magnitude lower bound for
\begin{align}
\label{eq:sketchp123claim}
\sum_{\substack{x-x^\alpha < p_1 p_2 p_3 \leq x \\ x^{\frac{r_j-1}{2K_j}} < p_j \leq x^{\frac{r_j+1}{2K_j}}}} \frac{f(p_1) \log p_1 f(p_2) \log p_2 f(p_3) \log p_3}{p_1 p_2 p_3}.
\end{align}

Recall $\Delta = x^{\alpha-1}/(K+3)$ and take $N = \frac{\log x}{\log(1+\Delta)}$. For $j = 1,2,3$, define $g_j \colon \mathbb{Z} \to \mathbb{R}_{\geq 0}$ by
\[
g_j(k) \coloneqq \frac{\mathbf{1}_{|k| < N/(2K_j)}}{\log(1+\Delta)} \sum_{\substack{x^{r_j/(2K_j)}(1+\Delta)^{k-1} < p_j \leq x^{r_j/(2K_j)} (1+\Delta)^{k} }} \frac{f(p_j) \log p_j}{p_j}.
\]
Then $k_1+k_2+k_3 \in \{-K+1, -K+2, \dotsc, 0\}$  implies (using $\frac{r_1}{2K_1} + \frac{r_1}{2K_2} +\frac{r_1}{2K_3} =1$) that for 
\[
p_j \in (x^{r_j/(2K_j)}(1+\Delta)^{k_j-1}, x^{r_j/(2K_j)} (1+\Delta)^{k_j}], \quad j=1,2,3,
\] 
we have $p_1p_2p_3 \in (x-x^\alpha, x]$. The task of lower bounding~\eqref{eq:sketchp123claim} is reduced to obtaining a lower bound for the triple convolution
\begin{align*}
\sum_{\ell \in \{-K+1, -K+2, \dotsc, 0\}} (g_1\ast g_2\ast g_3)(\ell) =\sum_{\substack{k_1 + k_2 + k_3 \in \{-K+1, \dotsc, 0\} \\ |k_j| \leq N/(2K_j)}} g_1(k_1)g_2(k_2)g_3(k_3)
\end{align*}
over the group $(\Z,+).$ 

By the assumption (A2*) the function $g_j(k)$ takes values in $[0,\kappa(\frac{r_j-1}{2K_j})]$ and by the assumption (A3*) has mean value $ \geq 1+o(1)$ on any interval $\mathcal{J} \subseteq [-N/(2K_j),N/(2K_j)]$ with $|\mathcal{J}| \geq N/K^3$. We obtain a lower bound for the triple convolution by using Fourier analysis on $(\Z,+)$ (i.e. the circle method). This group does have a subgroup of index 2, namely, $2\Z$, but since we are counting solutions to $k_1 + k_2 + k_3 \in \{-K+1, -K+2, \dotsc, 0\}$, we cover both cosets so that we do not need an assumption of the form (A4) here. Since $K=Q^4!$, also cosets of index $q \leq Q^4$ are covered, which gives a numerical improvement in the lower bound. We also took $K_2 = K Q$ large compared to $K_1=K_3=K$, so that $k_1+k_2$ cannot be a lot bigger than $k_3$, which helps us to avoid losses in the Fourier argument.

\subsection{Proof of Theorem \ref{thm_Linnik}}

Write $L = \Lv$. In order to deduce Theorem~\ref{thm_Linnik} from Theorem \ref{thm:non-torsion-special}, we take $\mathcal{N} = \{n \in \mathbb{N}\colon  (n, q) = 1\}$, $G=\mathbb{Z}_{q}^{\times}, f(n)=n\Mod q$, and $Y = q$. We shall apply Theorem \ref{thm:non-torsion-special} with $x = q^B$ for varying $B \in [L/2, L]$. 

With these choices, we will see that assumption (A1) holds with $\theta_B(\rho) = 1-1/(\rho B)-\varepsilon$ and $h(n) = \mathbf{1}_{(n, q)=1}$, and assumption (A2) holds with 
\[
\kappa_B = 2 + 2\frac{\log\left(\frac{B-3}{B-6}\right)}{\log 2} - \varepsilon.
\]
If also (A4) holds with an appropriate value of $\eta$ (i.e. for any index $2$ subgroup $H \leq G$ there is a fair proportion of primes in $[x^{1/6}, x^{1/3}]$ in $aH$), then Theorem \ref{thm:non-torsion-special} gives the desired conclusion.

We can then assume that for each $B \in [L/2,L]$ with $x = q^B$ there is a bias for some subgroup $H$. If this subgroup varies with $B$, then by continuity we can find some $B \in [L/2,L]$ such that there is bias for two different subgroups $H$. But this can be shown to be impossible due to orthogonality of characters.

We cannot rule out the scenario where the primes $p\in [x^{1/6}, x^{1/3}]$ are heavily biased in cosets of a subgroup of index $2$ (in particular, this happens if a Dirichlet $L$-function of modulus $q$ has a Siegel zero). However, if such a bias occurs for the same subgroup $H$ for all $x=q^B$ with $B \in [L/2,L]$, we can use a different sieve method to directly prove Linnik's theorem. Note that allowing $B$ to vary allows us to get a bias on a longer interval.

The fact that this scenario of (A4) not holding may be dealt with is natural from the fact that, under the existence of Siegel zeros, Linnik's theorem holds even with $L=2+o(1)$ by the work of Heath-Brown \cite{HB-Siegel-Linnik} (and one can even go beyond this assuming the existence of sufficiently strong Siegel zeros, as in the work of Friedlander and Iwaniec \cite{FI-Siegel}). The case of existence of Siegel zeros was handled separately using a different sieve method also in the work of Friedlander and Iwaniec~\cite[Chapter 24]{Opera} on Linnik's problem.

\begin{remark}
\label{rem:CompareLinnikProofs}
The classical proof of Linnik's theorem, presented e.g. in \cite[Chapter 18]{iw-kow}, requires three inputs on the zeros of Dirichlet $L$-functions: a zero-free region, zero repulsion, and a log-free zero density estimate. Our proof avoids all these ingredients --- the only property of $L$-functions used is the classical bound $L(1, \chi) \gg q^{-1/2}$ when $\chi \pmod{q}$ is a quadratic character. 

The proof of Friedlander and Iwaniec (presented in \cite[Chapter 24]{Opera} and in~\cite{FI1, FI2} with the value $L=75 \, 744 \, 000$) also avoided the use of zero density results and zero repulsion, but it does use the classical zero-free region for Dirichlet $L$-functions. The proof of Granville, Harper and Soundararajan \cite{GHS} avoided all three ingredients, but also probably leads to a rather large value of $L$.
\end{remark}

\subsection{Proof of Theorem~\ref{thm:Primesinshorts}}
\label{ssec:PrimesInShorts}
The deduction of Theorem~\ref{thm:Primesinshorts} from Theorem~\ref{thm:shorts-special} is easy and we do it already here.
\begin{proof}[Proof of Theorem~\ref{thm:Primesinshorts} assuming Theorem~\ref{thm:shorts-special}] Let $\varepsilon > 0$ be small.
We apply Theorem~\ref{thm:shorts-special} with 
\begin{align}
\label{eq:primeshortschoices}
f(n) = 1, \quad h(d)=1, \quad Y=1, \quad \theta=\alpha-\eps, \quad \text{and} \quad \kappa(\rho) = \frac{2\rho}{\rho-1+\alpha} + \eps.
\end{align}
Then the hypothesis (A1*) holds since trivially, for any $y,z > 0$ and $d \in \mathbb{N}$,
\begin{align} \label{eq:shortlevel}
\sum_{dn \in (y-z,y]} 1 = \frac{z}{d} + O(1).
\end{align}
Assumption (A2*) holds for $\kappa$ defined in~\eqref{eq:primeshortschoices} since by~\eqref{eq:shortlevel} the sequence $\{(1-\Delta)x^\rho < n \leq x^\rho\}_{n \leq x^\rho}$ has level of distribution $\frac{\rho-1+\alpha}{\rho}-\eps'$ for any $\eps'>0$, so that the upper bound follows from the linear sieve (Lemma~\ref{le:linear}), for instance. Finally, hypothesis (A3*) follows from Mertens' theorem $\sum_{p \leq z} \frac{\log p}{p} = \log z + O(1)$. Thus, by~\eqref{eq:FinalLowShorts} and Remark~\ref{rem:ShortsSimpleMT} we have
\begin{align*}
\sum_{\substack{x-x^\alpha < p \leq x}} 1 &\geq (\mathfrak{C}(\alpha) +o(1))\frac{x^{\alpha}}{\log x}
\end{align*}
with
\begin{align*}
\mathfrak{C}(\alpha) = \int_{(\theta-1/2)/3}^{1/3} \int_{\max\{\alpha_2, \theta-3\alpha_2\}}^{\min\{1/2, 1-2\alpha_2\}} \frac{\mathcal{F}^\sharp(\frac{1}{\kappa(\alpha_1)}, \frac{1}{\kappa(\alpha_2)}, \frac{1}{\kappa(1-\alpha_1-\alpha_2)})}{\alpha_1 \alpha_2 (1-\alpha_1 - \alpha_2)} \d\alpha_1 \d\alpha_2  - \frac{2}{\theta}\log \frac{1}{2\theta-1}.
\end{align*}
Computing the integral numerically\footnote{This and all other numerical calculations can be found in the Mathematica file which is provided as a supplement to the arXiv version of this paper.}, we see that $\mathfrak{C}(\alpha)> 1/40$ when $\alpha = \Hv$ and $\varepsilon$ is sufficiently small.
\end{proof}

\subsection{Further developments}
\label{ssec:future}
We mention here further applications that we are working on, which will appear elsewhere.

Theorem~\ref{thm:non-torsion-special} is stated in terms of the multiplicative function $f\colon  \mathcal{N} \to G$. It may be generalized so that products of three primes are directly related to a triple convolution in the group $G$. A generalization of Theorem~\ref{thm:non-torsion-special} will apply to give a Chebotarev density version of Linnik's theorem for an extension of number fields $L/K$, with the finite abelian group being $G= \mathrm{Gal}(L/K)$ (the case of a non-abelian extension may be reduced to the abelian case by a well known argument). 

Theorem~\ref{thm:shorts-special} applies to general multiplicative functions $f$, and we will apply it to the Fourier coefficients $|\lambda_\pi(n)|^2$ of a $GL_m$ automorphic representation along primes in short intervals $(x-x^\alpha, x]$.

In both of these applications we expect to get a good explicit numerical dependence on the conductor (of the relevant $L$-functions), but since we use sieve methods the dependence on the degree (of the relevant $L$-functions) is likely to be quite bad.

The more general version of Theorem~\ref{thm:non-torsion-special} will also give an elementary proof of a lower bound for Gaussian primes $\pi \in \Z[i], |\pi|^2 \leq x$, with argument in narrow sectors $\arg \pi \in [\theta,\theta+ x^{-\eps}]$. Here we use the compact group $G=\R/ 2 \pi \Z$ which can be approximated by finite abelian groups $\Z/N\Z$ with $N$ large.

We are also working on an extension of the sieve to replace the abelian group with conjugacy classes of a non-commutative group, to generalize both Theorems~\ref{thm:non-torsion-special} and~\ref{thm:shorts-special}. This would have applications to Sato--Tate laws.

Finally, it seems that by a more careful argument the proof of Theorem~\ref{thm:shorts-special} can be extended to give an asymptotic formula for the number of primes in the interval $(x-x^\alpha, x]$ for any $\alpha =1 - o(1)$, with an elementary proof. 

\section{Notation and preliminaries}

\subsection{Operations on abelian groups}

We shall need the notions of convolutions and Fourier transforms on a (multiplicatively written) abelian group $G$. For functions $f,g\colon G\to \mathbb{C}$, we define their convolution as
\begin{align*}
(f*g)(a)\coloneqq \sum_{\substack{a_1,a_2\in G\\a_1a_2=a}}f(a_1)g(a_2).   
\end{align*}

The elements of the character group $\widehat{G}$ of a finite abelian group $G$ are the group homomorphisms $\chi\colon G\to S^{1},$ where $S^1$ is the unit circle of the complex plane. One can show that $|\widehat{G}|=|G|$ and that we have the orthogonality relations
\begin{align*}
1_{g_1=g_2}=\frac{1}{|G|}\sum_{\chi\in \widehat{G}}\chi(g_1)\overline{\chi}(g_2),\quad \quad 
1_{\chi_1=\chi_2}=\frac{1}{|G|}\sum_{g\in G}\chi_1(g)\overline{\chi_2}(g).
\end{align*}
We say that $m\geq 1$ is the order of $\chi$ if $\chi^m\equiv 1$ and $\chi^k\not \equiv 1$ for all $1\leq k<m$. There exists a unique character $\chi_0$ whose order is $1$ and we say that $\chi_0$ is the trivial character.

For a function $f\colon G\to \mathbb{C}$, its Fourier transform $\widehat{f}\colon \widehat{G}\to \mathbb{C}$ is defined by
\begin{align*}
\widehat{f}(\chi)\coloneqq \mathbb{E}_{b \in G}f(b)\overline{\chi}(b). 
\end{align*}
With this normalization, Parseval's identity (which follows by expanding out the square and using the orthogonality relations) states that
\begin{align*}
\sum_{\chi\in \widehat{G}}|\widehat{f}(\chi)|^2=\mathbb{E}_{n\in G}|f(n)|^2.    
\end{align*}

We will use the following basic fact: for every subgroup $H$ of index 2 in $G$ there is a unique quadratic character $\psi \in \widehat{G}$ such that $\ker \psi = H$, that is, $\mathbf{1}_{b \in H} = \frac{1}{2}(1+\psi(b))$.

\subsection{Operations on \texorpdfstring{$\mathbb{Z}$}{Z}}
We shall need the notions of convolutions and Fourier transforms also on $\mathbb{Z}$. For clarity we shall in this set-up denote the Fourier transform by $\widecheck{f}$ instead of $\widehat{f}$.

For functions $f,g\colon \mathbb{Z} \to \mathbb{C}$ with finite support, we define their convolution as
\begin{align*}
(f*g)(a)\coloneqq \sum_{\substack{a_1,a_2\in \mathbb{Z}\\a_1+a_2=a}}f(a_1)g(a_2).   
\end{align*}
For a function $f\colon  \mathbb{Z} \to \mathbb{C}$ supported on $[-M, M]$, its Fourier transform $\widecheck{f}\colon  \mathbb{R} \to \mathbb{C}$ is defined by (the choice of $M$ will be clear from context)
\begin{align*}
\widecheck{f}(\theta)\coloneqq \mathbb{E}_{|b| \leq M} f(b)e(b\theta). 
\end{align*}
With this normalization, Parseval's identity states that
\begin{align*}
\int_0^1 |\widecheck{f}(\theta)|^2 \d\theta = \mathbb{E}_{|n| \leq M} |f(n)|^2.    
\end{align*}

\subsection{Sieve theory notation}
\label{ssec:Sieve}

We use the standard sieve notation (introduced e.g. in~\cite[Chapter 6]{Opera}). Thus, for a nonnegative sequence $(a_n)$ and $d\in \mathbb{N}$, write
\begin{align*}
 \mathcal{A}\coloneqq  (a_n)_n,\quad \mathcal{A}_d\coloneqq  (a_{dn})_{n},\quad |\mathcal{A}_d|\coloneqq  \sum_{n}a_{dn}.   
\end{align*}
Also, for $z\geq 1$ we write $P(z)=\prod_{p<z}p$ and 
\[
S(\mathcal{A},z) \coloneqq   \sum_{\substack{n\\(n,P(z))=1}} a_n.
\]

We shall use the linear sieve several times, so we next introduce the definitions and results we need. The linear sieve functions $f_1,F_1\colon \mathbb{R}_+\to \mathbb{R}$ are the unique functions such that $f_1$ is continuously differentiable on $(2,\infty)$, $F_1$ is continuously differentiable on $(1,\infty)$ and such that $f_1$ and $F_1$ satisfy the pair of delayed differential equations 
\begin{align*}
&\frac{\d}{\d s}(sF_1(s))=f_1(s-1),\quad s>3\\
&\frac{\d}{\d s}(sf_1(s))=F_1(s-1),\quad s>2    
\end{align*}
with the initial conditions $F_1(s)=\frac{2e^{\gamma}}{s}$ for $s\in (0,3]$ and $f_1(s)=0$ for $s\in (0,2]$. In particular $f_1(s) = \frac{2 e^\gamma}{s}\log(s-1)$ for $s \in [2, 4]$.

The upper and lower linear sieve coefficients $\lambda_{d}^{\pm}$ of level $D$ and sifting parameter $z$ are given by
\begin{align}
\label{eq:lamdpmdef}
\lambda_{d}^{\pm}=\mu(d)1_{d\in \mathcal{D}^{\pm}},    
\end{align}
where
\begin{align*}
\mathcal{D}^{+}=\{d\in \mathbb{N}:&\,\, d=p_1\cdots p_r\textnormal{ for some } p_r<\cdots <p_1<z\\
&\textnormal{ with } p_1\cdots p_{m-1}p_m^3<D\, \, \forall m\leq r,\, m \equiv 1\Mod{2}\},
\end{align*}
and $\mathcal{D}^{-}$ is defined analogously but with the congruence condition $m\equiv 0\Mod 2$ instead.

\begin{lemma}[Linear sieve] \label{le:linear}
Let $z\geq 2$ and $D=z^{s}$ with $s\geq 2$. Let $\lambda_d^{+}, \lambda_d^{-}$ be the upper and lower linear sieve coefficients, respectively, with level $D$ and sifting parameter $z$. 
\begin{enumerate}[(i)]
    \item For any $n\geq 1$, we have
    \begin{align*}
       \sum_{d\mid n}\lambda_d^{-}\leq  1_{(n, P(z))=1} \leq \sum_{d\mid n}\lambda_d^{+}.
    \end{align*}
    \item Let $h\colon \mathbb{N}\to \mathbb{R}_{\geq 0}$ be any multiplicative function satisfying $h(p)<p$ for every prime $p$ and such that for some fixed $C\geq 1$ and all $w_2\geq w_1\geq 2$ we have 
    \begin{align*}
    \prod_{w_1\leq p<w_2}\left(1-\frac{h(p)}{p}\right)^{-1}\leq \left(1+\frac{C}{\log w_1}\right)\frac{\log w_2}{\log w_1}.    
    \end{align*}
    Then
    \begin{align}
    \nonumber
     \sum_{d\mid P(z)}\lambda_d^{+}\frac{h(d)}{d}&\leq (F_1(s)+O((\log D)^{-1/6}))\prod_{p<z}\left(1-\frac{h(p)}{p}\right),\\ 
     \label{eq:lamb-}
     \sum_{d\mid P(z)}\lambda_d^{-}\frac{h(d)}{d}&\geq (f_1(s)+O((\log D)^{-1/6}))\prod_{p<z}\left(1-\frac{h(p)}{p}\right),  
    \end{align}
    where $f_1$ and $F_1$ are the linear sieve functions, defined above.
\end{enumerate}
\end{lemma}

\begin{proof}
See e.g. \cite[Section 12.1]{Opera}.
\end{proof}

\section{The sieve}

We shall use a consequence of an exact form of the linear sieve, claiming back the nonnegative terms $S_r$ that the linear sieve throws away. A similar idea was used by Friedlander and Iwaniec in their proof of Linnik's theorem in~\cite[Chapter 24]{Opera}.

In the following lemma and later we write, for a sequence $\mathcal{B} = (b_n)_{n \in \mathbb{N}},$ parameters $D, w > 1$, and $r \in \mathbb{N}$,
\[
S_r(\mathcal{B}, D, w) \coloneqq   \sum_{\substack{p_r < \dotsb < p_1 < w \\ p_1 \dotsm p_m p_m^2 < D \text{ for $m < r, m \equiv r \Mod{2}$} \\ p_1 \dotsm p_r p_r^2 \geq D}} S(\mathcal{B}_{p_1 \dotsm p_r}, p_r).
\]

\begin{lemma} \label{le:lowerlinear}
Let $x \geq 2$, let $\mathcal{A} = (a_n)_{n \leq x}$ be a sequence of nonnegative real numbers that has level of distribution $\theta \in (0, 1]$ with dimension $1$, size $X > 0$ and density $h \colon \mathbb{N} \to \mathbb{R}_{\geq 0}$. Write $D = x^\theta$ and let $z = x^u$ with $u \in [\theta/2, \theta]$. Then
\begin{align*}
S(\mathcal{A}, z) &\geq - X e^{\gamma} \int_{\theta/2}^u \prod_{p < x^{(\theta-\beta)/2}} \left(1-\frac{h(p)}{p}\right)\frac{\d\beta}{\beta} + \sum_{r \text{\emph{ even}}} S_r(\mathcal{A}, D, z) + o\left(X \prod_{p < x} \left(1-\frac{h(p)}{p}\right)\right).
\end{align*}
\end{lemma}

\begin{proof}
Let $\lambda_d^-$ be as in~\eqref{eq:lamdpmdef} with $z$ and $D$ as in the statement of Lemma~\ref{le:lowerlinear}. By definitions we have (see e.g.~\cite[Section 12.1]{Opera})
\[
S(\mathcal{A}, z) = \sum_{d \mid P(z)} \lambda_d^- |\mathcal{A}_d| + \sum_{r \text{ even}} S_r(\mathcal{A}, D, z).
\]
Let $\lambda_{p, d}^+$ be as in~\eqref{eq:lamdpmdef} but with $p$ in place of $z$ and $D/p$ in place of $D$. Then by definitions
\begin{align*}
\sum_{d \mid P(z)} \lambda_d^- |\mathcal{A}_d| &= \sum_{d \mid P(D^{1/2})} \lambda_d^- |\mathcal{A}_d| + \sum_{D^{1/2} < p < z} \sum_{d \mid P(p)} \lambda_{dp}^- |\mathcal{A}_{dp}| \\
&= \sum_{d \mid P(D^{1/2})} \lambda_d^- |\mathcal{A}_d| - \sum_{D^{1/2} < p < z} \sum_{d \mid P(p)} \lambda_{p, d}^+ |\mathcal{A}_{dp}|.
\end{align*}
Now consider the contribution to the second term on the right-hand side of those $d$ which have a prime factor $\in [(D/p)^{1/2}, p)$. Let $p_0$ be largest such factor. Looking at the definition of $\lambda_{p,d}^+$ and taking $n = 1$ there, we see that we must have $p_0^3 < D/p$ which is actually not possible. Hence we have  
\begin{align*}
S(\mathcal{A}, z) &= \sum_{d \mid P(D^{1/2})} \lambda_d^- |\mathcal{A}_d| - \sum_{D^{1/2} < p < z} \sum_{d \mid P((D/p)^{1/2})} \lambda_{p, d}^+ |\mathcal{A}_{dp}| + \sum_{r \text{ even}} S_r(\mathcal{A}, D, z) \\
&=\sum_{d \mid P(D^{1/2})} \lambda_d^- X \frac{h(d)}{d} - \sum_{D^{1/2} < p < z} \sum_{d \mid P((D/p)^{1/2})} \lambda_{p, d}^+ X \frac{h(dp)}{dp} + \sum_{r \text{ even}} S_r(\mathcal{A}, D, z) \\
&\qquad + O\left(\sum_{d\leq D} \mu^2(d) \left||\mathcal{A}_d| - \frac{h(d)}{d}X\right|\right).
\end{align*}
The error term can be dealt with the level of distribution assumption. By the linear sieve lower and upper bounds (Lemma~\ref{le:linear}(ii) with $s = 2$) we obtain
\begin{align*}
S(\mathcal{A}, z) &\geq (f_1(2)+o(1)) X \prod_{p < D^{1/2}} \left(1-\frac{h(p)}{p}\right) \\
& \qquad - X \sum_{D^{1/2} < p < z} \frac{h(p)}{p} (F_1(2)+o(1)) \prod_{p' < (D/p)^{1/2}} \left(1-\frac{h(p')}{p'}\right) + \sum_{r \text{ even}} S_r(\mathcal{A}, D, z) .
\end{align*}
The claim follows since $f_1(2) = 0$ and $F_1(2) = e^{\gamma}$, and since the assumption on $h$ implies that for any $w_2\geq w_1\geq 2$ we have
\begin{align*}
\sum_{w_1\leq p<w_2}\frac{h(p)}{p}\leq \log \frac{\log w_2}{\log w_1}+O\left(\frac{1}{\log w_1}\right).
\end{align*}
\end{proof}

\section{Linear combinations of roots of unity}
We will bound some Fourier transforms appearing later in Sections~\ref{sec:tripleconvolutionZ} and~\ref{sec: convolution} by using the following lemma.
\begin{lemma}
\label{le:LinCombRootsOfUnity} Let $\delta \in [0, 1]$ and $m \geq 3$. For each $b \in \{0, 1, \dotsc, m-1\}$, let $w_b \in [0, 1/m]$. Assume that $\sum_{b =0}^{m-1} w_b = \delta$. Then
\begin{align*}
 \left|\sum_{b=0}^{m-1} e\left(\frac{b}{m}\right) w_b \right| \leq S(\delta, m),
\end{align*}
where 
\begin{align}
\label{eq:Sdelmdef}
S(\delta, m) \coloneqq   \left|\frac{\sin\left(\pi \frac{\left\lfloor \delta m \right\rfloor}{m}\right)}{m \sin\left(\frac{\pi}{m}\right)} + \left(\delta-\frac{\lfloor \delta m\rfloor}{m}\right) e\left(\frac{\lfloor \delta m \rfloor+1}{2m} \right)\right|.
\end{align}
\end{lemma}

\begin{proof}
This is similar to~\cite[Lemma 5.9]{tera-polynomial}; for the sake of completeness we give details. Note that there is some real number $\theta$ such that 
\begin{align*}
 \left|\sum_{b=0}^{m-1} e\left(\frac{b}{m}\right)w_b\right|= \textnormal{Re}\left(e(\theta)\sum_{b=0}^{m-1} e\left(\frac{b}{m}\right)w_b\right)=\sum_{b=0}^{m-1}w_b\cos\left(2\pi\left(\theta+\frac{b}{m}\right)\right)\coloneqq \sum_{b=0}^{m-1}w_bx_b.  
\end{align*}
Let $k = \lfloor \delta m \rfloor$. Then the largest $k$ values of $x_b$ are attained for $b\in \mathcal{J}\pmod m$, where $\mathcal{J}=\{a+1,\ldots, a+k\}$ is some interval of length $k$. Moreover, the $k+1$st largest value is either $x_a$ or $x_{a+k+1}$; suppose for example that it is $x_{a+k+1}$.  Note that since $w_j\in [0,1/m]$,  the quantity $\sum_{b=0}^{m-1}w_bx_b$ does not decrease if for indices $j,r$ satisfying $x_j\geq x_r$ we repeatedly make transformations of the form $w_j\mapsto w_j+c$, $w_r\mapsto w_r-c$, with $0\leq c\leq \min\{1/m-w_j,w_r\}$. After performing finitely many such transformations, we see that
\begin{align*}
 \sum_{b=0}^{m-1}w_bx_b&\leq \sum_{b\in \mathcal{J}}\frac{1}{m}x_b+\left(\delta-\frac{k}{m}\right)x_{a+k+1}\\
 &\leq  \left|\sum_{b=a+1}^{a+k} \frac{1}{m} e\left(\frac{b}{m}\right)+\left(\delta-\frac{k}{m}\right)e\left(\frac{a+k+1}{m}\right)\right|\\ 
 &=\left|\sum_{0\leq b < \lfloor \delta m\rfloor}\frac{1}{m}e\left(\frac{b}{m}\right)+\left(\delta-\frac{\lfloor \delta m\rfloor}{m}\right)e\left(\frac{\lfloor \delta m\rfloor}{m}\right)\right|.
\end{align*}
Summing the geometric series, we see that the last expression is
\begin{equation}
\label{eq:gupSfirst}
\left|\frac{e\left(\frac{\lfloor \delta m\rfloor}{m}\right)-1}{m(e\left(\frac{1}{m}\right)-1)}+\left(\delta-\frac{\lfloor \delta m\rfloor}{m}\right)e\left(\frac{\lfloor \delta m\rfloor}{m}\right)\right|.
\end{equation}

Since for any $\beta \in \mathbb{R}$ we have
\[
e(\beta)-1 = e\left(\frac{\beta}{2}\right)\left(e\left(\frac{\beta}{2}\right) -  e\left(-\frac{\beta}{2}\right)\right) = e\left(\frac{\beta}{2}\right) \cdot 2i \cdot \sin(\pi \beta),
\]
the claim follows by applying this to both the numerator and the denominator of the first term inside the absolute values in~\eqref{eq:gupSfirst} (multiplying also both sides by $1 = |e(\frac{-\lfloor \delta m \rfloor+1}{2m})|$).
\end{proof}

From this and the Taylor approximation of the sine function, we immediately  get the following asymptotic formula for $S(\delta, m)$.
\begin{lemma} \label{le:Smlimit} Let $\delta\in [0, 1]$, and let $S(\delta, m)$ is as in~\eqref{eq:Sdelmdef}. Then for $m \to \infty$
\begin{align*}
 S(\delta, m) =  \frac{\sin(\pi \delta)}{\pi } + O\left(\frac{1}{m}\right). 
\end{align*}
\end{lemma}

\section{A triple convolution estimate in \texorpdfstring{$\mathbb{Z}$}{Z}}\label{sec:tripleconvolutionZ}
\begin{proposition} \label{prop:tripconvZ}
Let $\delta_1, \delta_2, \delta_3 \in (0,1/2)$ be fixed and let $M \in \mathbb{N}$. Let $K=Q_0!$ for some $Q_0 \in \mathbb{N}$ such that $K \leq M^{1/4}$, let $\nu \coloneqq   Q_0^{-1/2}$, and let $\beta \in (Q_0^{-1/4}, 1]$. 

For $j=1, 2, 3$, let $h_j\colon  \mathbb{Z} \to [0,1]$ be such that $h_2$ is supported on $[-\beta M,\beta M]$ and $h_1, h_3$ are supported on $[-M,M]$. 

Assume also that
\begin{align} \label{eq:h23assump}
\E_{k \in (r\nu^3 M, (r+1)\nu^3M]} h_2(k) = \delta_2+O(\nu^2)
\end{align}
whenever $|r| \leq \beta \nu^{-3}-1$ and, for $j \in \{1, 3\}$,
\begin{align} \label{eq:h1assump}
\E_{k \in (r\nu^3 M, (r+1)\nu^3M]} h_j(k) = \delta_j+O(\nu^2),
\end{align}
whenever $|r| \leq \nu^{-3}-1$.
Then
\begin{align*}
&\sum_{\ell  \in \{0,1,\dots,K-1\}}(h_1 \ast h_2 \ast h_3) (\ell)  \\
&\geq \beta K (2M+1)^2 \left(\delta_1 \delta_2 \delta_3 \left(1-\frac{\beta}{4}\right) -\frac{\sin(\pi \delta_2)}{\pi}\sqrt{(\delta_1-\delta_1^2)(\delta_3-\delta_3^2)} + O(Q_0^{-1/4})\right). 
\end{align*}
\end{proposition}
When $\beta > 0$ is small, the lower bound is $\gg \beta K M^2$ when $\delta_1, \delta_2, \delta_3 \geq 0.43$.

We start by providing an $L^{\infty}$ bound which we will use to bound the Fourier transform of $h_2$ in minor arcs.
\begin{lemma}
\label{le:LinftyZ}
Let $\delta \in [0, 1/2]$ be fixed and let $N \geq 1$. Assume that $h \colon \mathbb{Z} \to [0,1]$ is supported on $[-N, N]$ and satisfies
\begin{equation}
\label{eq:h(n)av}
\mathbb{E}_{|n| \leq N} h(n) = \delta.
\end{equation}
Let $a, q \in \mathbb{N}$ and $Q_1, Q_2 \in \mathbb{R}$ satisfy $N \geq Q_2 \geq q \geq Q_1 \geq 1$ and $(a, q) = 1$. Then, for any $\theta \in \mathbb{R}$ with 
\begin{equation}
\label{eq:thetaapprox}
\left|\theta-\frac{a}{q}\right| \leq \frac{1}{qQ_2},
\end{equation} 
we have
\begin{align} \label{eq:minorarcbound}
|\widecheck{h} (\theta)| = \left| \mathbb{E}_{|n| \leq N} h(n) e(n\theta) \right| \leq  \frac{\sin(\pi \delta)}{\pi} +  O\left(\frac{N}{Q_1 Q_2} + \frac{Q_2}{N} + \frac{1}{Q_1}\right)
\end{align}
\end{lemma}

\begin{proof}
We have using~\eqref{eq:thetaapprox}
\[
\widecheck{h} (\theta) = \widecheck{h} (a/q) +  O\left(\frac{N}{q Q_2}\right).
\]
Here we write
\[
\widecheck{h} (a/q) =  \sum_{b \Mod{q}} e(b/q) w_b, \quad \text{with} \quad w_b = \E_{|k| \leq N} h(k) \mathbf{1}_{k \equiv b \overline{a} \, \Mod{q}}.
\]
Now we have the trivial bounds $0\leq w_b\leq (1+O(q/N))/q$, and by \eqref{eq:h(n)av} we get $\sum_{b} w_b = \delta$. Hence, by Lemmas~\ref{le:LinCombRootsOfUnity} and~\ref{le:Smlimit}
\[
|\widecheck{h} (a/q)| \leq \left(\frac{\sin(\pi \delta)}{\pi} + O\left(\frac{1}{Q_1}\right)\right)\left(1+O\left(\frac{q}{N}\right)\right) = \frac{\sin(\pi \delta)}{\pi} + O\left(\frac{q}{N} + \frac{1}{Q_1}\right)
\]
\end{proof}

\begin{proof}[Proof of Proposition~\ref{prop:tripconvZ}]
For any $\ell \in \mathbb{Z}$ we have
\[
(h_1 \ast h_2 \ast h_3) (\ell)  = (2M+1)^3 \int_{0}^1 \widecheck{h}_1 (\theta) \widecheck{h}_2 (\theta) \widecheck{h}_3 (\theta) e(-\ell\theta ) \d \theta,
\]
where $\widecheck{h}_j(\theta) = \mathbb{E}_{|k| \leq M} h_j(k) e(k\theta)$. Recall $\nu = Q_0^{-1/2}$  and let $Q= \nu M$. We write (by Dirichlet's approximation theorem)
\[
[0,1] = \mathfrak{M} \cup  \mathfrak{m}_1 \cup \mathfrak{m}_2  \pmod{1},
\]
where modulo 1 we have
\begin{align*}
\mathfrak{M}  &\coloneqq   \bigg[ -\frac{1}{Q},\frac{1}{Q} \bigg], \quad
\mathfrak{m}_1  \coloneqq  \bigcup_{\substack{2 \leq q \leq Q_0 \\ (a,q)=1}} \bigg[ \frac{a}{q} - \frac{1}{qQ},  \frac{a}{q} +\frac{1}{qQ} \bigg], \\
\mathfrak{m}_2 &\coloneqq   [0, 1] \setminus \left(\mathfrak{M} \cup \mathfrak{m}_1\right) \subseteq \bigcup_{\substack{Q_0 < q \leq Q \\ (a,q)=1}} \bigg[ \frac{a}{q} - \frac{1}{qQ},  \frac{a}{q} +\frac{1}{qQ} \bigg].
\end{align*}
The major arc $\mathfrak{M}$ will contribute the main term while the minor arcs $\mathfrak{m}_1$ and $\mathfrak{m}_2$ only contribute to the error term.
\subsection*{Contribution from  \texorpdfstring{$\mathfrak{M}$}{M}}
For $\theta \in \mathfrak{M}$ we have $|\theta| \leq \frac{1}{Q} = \frac{1}{\nu M}$. By using \eqref{eq:h23assump} we get
\begin{align*} 
\widecheck{h}_2 (\theta) &=\E_{|k| \leq M} h_2(k) e(k \theta ) =\frac{1}{2M+1} \sum_{|j| < \beta \nu^{-3}-1} \sum_{k \in (j \nu^3 M, (j+1) \nu^3 M]} h_2 (k) e(k \theta ) + O(\nu^3) \\ 
&=\frac{1}{2M+1} \sum_{|j| < \beta \nu^{-3}-1} \sum_{k \in (j \nu^3 M, (j+1) \nu^3 M]} h_2(k) \left( e(j\nu^3 M\theta) + O\left(\nu^3 M \cdot \frac{1}{\nu M}\right)\right)  + O(\nu^3) \\ 
&= \frac{\delta_2 \nu^3 M}{2N+1} \sum_{|j| < \beta \nu^{-3}-1}  e(j \nu^3 M \theta ) +O(\nu^2) \\
&= \delta_2 \widecheck{\mathbf{1}}_{[-\beta M, \beta M]}(\theta) + O(\nu^2)
\end{align*}
and similarly by~\eqref{eq:h1assump}
\begin{align} \label{eq:h23onM1} 
\widecheck{h}_j (\theta) &= \delta_j \widecheck{\mathbf{1}}_{[-M, M]}(\theta) +  O(\nu^2)  
\end{align}
for $j = 1, 3$. Thus, we have for $\ell \in \{0,1,\dots,K-1\}$
\begin{align*}
&(2M+1)^3 \int_{\mathfrak{M}} \widecheck{h}_1 (\theta) \widecheck{h}_2 (\theta) \widecheck{h}_3 (\theta) e(- \ell \theta) \d \theta \\
&= \delta_1 \delta_2 \delta_3 (2M+1)^3 \int_{\mathfrak{M}} \widecheck{\mathbf{1}}_{[-\beta M, \beta M]} (\theta) \widecheck{\mathbf{1}}_{[-M, M]} (\theta) \widecheck{\mathbf{1}}_{[-M, M]} (\theta)e(- \ell \theta) \d \theta + O(\nu^3 M^3/Q).
\end{align*}
Since for example
\begin{equation}
\label{eq:1checkbound}
\widecheck{\mathbf{1}}_{[-M, M]} (\theta) \ll \min\{1, 1/(M \theta)\},
\end{equation}
we can extend the integration to $[0, 1]$, obtaining 
\begin{align*}
&(2M+1)^3 \int_{\mathfrak{M}} \widecheck{h}_1 (\theta) \widecheck{h}_2 (\theta) \widecheck{h}_3 (\theta) e(- \ell \theta) \d \theta \\
&= \delta_1 \delta_2 \delta_3 \sum_{\substack{k_1 + k_2 + k_3 = \ell \\ |k_2| \leq \beta M \\ |k_1|, |k_3| \leq M}} 1 + O(\nu M^2) = \delta_1 \delta_2 \delta_3 \sum_{|k_2| \leq \beta M} \sum_{\substack{|k_1| \leq M \\ |k_1 + k_2 - \ell| \leq M}} 1 + O(\nu M^2) \\
&=\delta_1 \delta_2 \delta_3\sum_{|k_2| \leq \beta M} (2M-|k_2|) + O(\nu M^2) = (\delta_1 \delta_2 \delta_3 + O(Q_0^{-1/4}))\beta (1-\beta/4) (2M)^2.
\end{align*} 
\subsection*{Contribution from  \texorpdfstring{$\mathfrak{m}_1$}{m1}}
We want to estimate
\begin{align*}
(2M+1)^3 \int_{\mathfrak{m}_1} \widecheck{h}_1 (\theta) \widecheck{h}_2 (\theta) \widecheck{h}_3 (\theta) \bigg(\sum_{\ell =0}^{K-1} e(-\ell\theta) \bigg)  \d \theta.
\end{align*}
Recall that $K=Q_0!$. Then for $\theta \in \mathfrak{m}_1$ with $|\theta-a/q| < 1/(qQ)$, $q \leq Q_0$,  we have
\begin{align*}
\sum_{\ell =0}^{K-1} e(-\ell\theta) = \sum_{\ell =0}^{Q_0!-1} e\bigg(-\ell\frac{a}{q}\bigg) +  O\bigg( \frac{K^2}{Q}\bigg) = O\bigg( \frac{K^2}{Q}\bigg). 
\end{align*}
By using the trivial bounds $|\widecheck{h}_1|,|\widecheck{h}_2| ,|\widecheck{h}_3|  \leq 1,$ and $|\mathfrak{m}_1| = 1/Q$ we get
\begin{align*}
(2M+1)^3  \bigg|\int_{\mathfrak{m}_1} \widecheck{h}_1 (\theta) \widecheck{h}_2 (\theta) \widecheck{h}_3 (\theta)\bigg(\sum_{\ell =0}^{K-1} e(-\ell\theta) \bigg)  \d \theta \bigg| \ll \frac{K^2 M^3}{Q^2} \ll \frac{K M^2}{Q_0} .
\end{align*}

\subsection*{Contribution from  \texorpdfstring{$\mathfrak{m}_2$}{m2}}
Using~\eqref{eq:h23onM1},~\eqref{eq:1checkbound}, and Parseval's identity, we see that, for $j= 1, 3$,
\begin{align} \label{eq:h23onM1asymp}
\int_{\mathfrak{M}}  |\widecheck{h}_j (\theta)|^2 \d\theta = \delta_j^2 \int_0^1 |\widecheck{\mathbf{1}}_{[-M, M]}(\theta)|^2 = \delta_j^2 \frac{1}{2M+1} + O(\nu/M).
\end{align}
Furthermore by Lemma~\ref{le:LinftyZ} with $N = \lfloor \beta M \rfloor, Q_1 = Q_0,$ and $Q_2 = Q = Q_0^{-1/2} M$, we see that
\[
\sup_{\theta \in \mathfrak{m}}  |\widecheck{h}_2 (\theta)| \leq \frac{2\beta M +1}{2M+1} \left(\frac{\sin (\pi \delta_2)}{\pi} + O\left(Q_0^{-1/4}\right)\right).
\]

Applying also the Cauchy-Schwarz inequality, Parseval's identity, and \eqref{eq:h23onM1asymp} we obtain
\begin{align*}
& \int_{\mathfrak{m}_2} |\widecheck{h}_1 (\theta) \widecheck{h}_2 (\theta) \widecheck{h}_3 (\theta) | \d \theta  \\
\leq & \frac{2\beta M+1}{2M+1} \cdot \left(\frac{\sin(\pi \delta_2)}{\pi}    + O(Q_0^{-1/4})\right) \bigg(\int_{\mathfrak{m}_2} |\widecheck{h}_1 (\theta)|^2 \d \theta\bigg)^{1/2}  \bigg(\int_{\mathfrak{m}_2} |\widecheck{h}_3 (\theta)|^2 \d \theta\bigg)^{1/2}  \\
&\leq \left( \beta \cdot \frac{\sin(\pi \delta_2)}{\pi}    + O(\beta Q_0^{-1/4})\right) \prod_{j \in \{1, 3\}} \bigg(\int_{[0,1]}  |\widecheck{h}_j (\theta)|^2  \d \theta- \int_{\mathfrak{M}}  |\widecheck{h}_j (\theta)|^2 \d \theta\bigg)^{1/2} \\
&=  \beta \cdot \frac{\sin(\pi \delta_2)}{\pi} \frac{1}{2M+1} \sqrt{(\delta_1-\delta_1^2)(\delta_3-\delta_3^2)} + O\left(\beta \frac{Q_0^{-1/4}}{M}\right).
\end{align*}
Combining the above bounds we get
\begin{align*}
&\sum_{\ell  \in \{0,1,\dots,K-1\}}(h_1 \ast h_2 \ast h_3) (\ell) \\
&\geq \beta K (2M+1)^2 \left( \delta_1 \delta_2 \delta_3 (1-\beta/4) - \frac{\sin(\pi \delta_2)}{\pi} \sqrt{(\delta_1-\delta_1^2)(\delta_3-\delta_3^2)} + O(Q_0^{-1/4})\right),
\end{align*}
as claimed.
\end{proof}

\section{Proof of Theorem~\ref{thm:shorts-special}} \label{sec:shortspecialproof}

Write $\mathcal{A} = (f(n))_{x-x^\alpha < n \leq x}$. Applying Lemma~\ref{le:lowerlinear} and for $r \geq 4$ using the trivial lower bound $S_r(\mathcal{A}, x^\theta, x^{1/2})  \geq 0$, it suffices to show that 
\begin{align*}
&S_2(\mathcal{A}, x^\theta, x^{1/2})\\
&\geq \frac{x^{\alpha}}{\log x} \left(\int_{(\theta-1/2)/3}^{1/3} \int_{\max\{\alpha_2, \theta-3\alpha_2\}}^{\min\{1/2, 1-2\alpha_2\}} \frac{\mathcal{F}^\sharp\left(\frac{1}{\kappa(\alpha_1)}, \frac{1}{\kappa(\alpha_2)}, \frac{1}{\kappa(1-\alpha_1-\alpha_2)}\right)}{\alpha_1 \alpha_2 (1-\alpha_1 - \alpha_2)} \d\alpha_1 \d\alpha_2 +o_{Q \to \infty}(1)\right).
\end{align*}
By definition
\[
S_2(\mathcal{A}, x^\theta, x^{1/2}) \geq \sum_{\substack{x-x^\alpha < p_1 p_2 p_3 \leq x \\ p_2 < p_1 < x^{1/2}, \, p_3 > p_2 \\ p_1 p_2^3 \geq x^\theta}} f(p_1) f(p_2) f(p_3).
\]

Recall $Q \in \mathbb{N}$ and $K = Q^4!$. We define $K_1 = K_3= K$ and $K_2 = K Q$ and split $p_j$ into intervals $p_j \in (x^{\frac{r_j-1}{2K_j}}, x^{\frac{r_j+1}{2K_j}}]$, with $(r_1, r_2, r_3)$ running over the set
\begin{align*}
\mathcal{R} \coloneqq   \Biggl\{(r_1, r_2, r_3) \in \mathbb{N}^3 \colon & \frac{r_1}{2K_1} + \frac{r_2}{2K_2} + \frac{r_3}{2K_3} = 1, \, \frac{r_2+1}{2K_2} < \frac{r_1-1}{2K_1}, \\ &\frac{r_1+1}{2K_1} < \frac{1}{2}, \, \frac{r_2+1}{2K_2} < \frac{r_3-1}{2K_1}, \, \frac{r_1-1}{2K_1} + 3\frac{r_2-1}{2K_2} > \theta\Biggr\}
\end{align*}
obtaining
\begin{align}
\label{eq:ffflow}
\begin{aligned}
&S_2(\mathcal{A}, x^\theta, x^{1/2}) \geq \sum_{\substack{x-x^\alpha < p_1 p_2 p_3 \leq x \\ p_2 < p_1 < x^{1/2}, \, p_3 > p_2 \\ p_1 p_2^3 \geq x^\theta}} f(p_1) f(p_2) f(p_3) \\
& \geq \sum_{(r_1, r_2, r_3) \in \mathcal{R}} \frac{x}{\frac{r_1+1}{2K_1} \frac{r_2+1}{2K_2} \frac{r_3+1}{2K_3} \log^3 x}\sum_{\substack{x-x^\alpha < p_1 p_2 p_3 \leq x \\ x^{\frac{r_j-1}{2K_j}} < p_j \leq x^{\frac{r_j+1}{2K_j}}}} \frac{f(p_1) \log p_1 f(p_2) \log p_2 f(p_3) \log p_3}{p_1 p_2 p_3}.
\end{aligned}
\end{align}
Consider a fixed triple $(r_1 ,r_2, r_3)$ counted above and recall that 
\begin{align*}
\Delta = x^{\alpha-1}/(K+3).
\end{align*} 
For $j = 1,2,3$, we define $g_j \colon \mathbb{Z} \to \mathbb{R}_{\geq 0}$ by
\[
g_j(k) \coloneqq   \frac{\mathbf{1}_{|k| < N/(2K_j)}}{\log(1+\Delta)} \sum_{\substack{x^{r_j/(2K_j)}(1+\Delta)^{k-1} < p \leq x^{r_j/(2K_j)} (1+\Delta)^{k} \\ x^{\frac{r_j-1}{2K_j}} < p \leq x^{\frac{r_j+1}{2K_j}}}} \frac{f(p) \log p}{p}.
\]
Write 
\[
N = 2KQ \left\lfloor \frac{\log x}{2KQ \log(1+\Delta)} \right\rfloor,
\]
so that $\frac{N}{2K_j} \in \Z$ for $j =1,2,3$. By (A2*) we have
\[
g_j(k) \leq \kappa\left(\frac{r_j-1}{2K_j}\right)
\]
for every $k$ and by (A3*) we have
\[
\mathbb{E}_{|k| < N/(2K_j)} g(k) = 1+o(1).
\]
Notice that, for $\ell \in \{-K+1, -K+2, \dotsc, 0\}$ and large enough $x$, we have 
\[
(x(1+\Delta)^{\ell-2}, x(1+\Delta)^{\ell}] \subseteq \bigg(x- \frac{K+2}{K+3}x^{\alpha} + O(K x^{2\alpha-1}), x\bigg] \subseteq (x-x^{\alpha},x]
\]
Hence
\begin{align}
\label{eq:p1p2p2low}
\begin{aligned}
&\sum_{\substack{x-x^\alpha < p_1 p_2 p_3 \leq x \\ x^{\frac{r_j-1}{2K_j}} < p_j \leq x^{\frac{r_j+1}{2K_j}}}} \frac{f(p_1) \log p_1 f(p_2) \log p_2 f(p_3) \log p_3}{p_1 p_2 p_3} \\
&\geq \log^3(1+\Delta) \sum_{\substack{k_1 + k_2 + k_3 \in \{-K+1, -K+2,\dotsc,0\} \\ |k_j| \leq N/(2K_j)}} g(k_1)g(k_2)g(k_3).
\end{aligned}
\end{align}

Let us now apply Proposition~\ref{prop:tripconvZ} with $h_j(n) = g_j(-n)/\kappa\left(\frac{r_j-1}{2K_j}\right), M = N/(2K)$, and $\beta = 1/Q$. We obtain
\begin{align*}
&\sum_{\substack{k_1 + k_2 + k_3 \in \{-K+1, -K+2, \dotsc, 0\} \\ |k_j| \leq N/(2K_j)}} g(k_1)g(k_2)g(k_3) \\
&\geq \frac{N^2}{Q K} \left(\mathcal{F}^\sharp\left(\frac{1}{\kappa\left(\frac{r_1-1}{2K_1}\right)}, \frac{1}{\kappa\left(\frac{r_2-1}{2K_2}\right)}, \frac{1}{\kappa\left(\frac{r_3-1}{2K_3}\right)} \right) + O\left(\frac{1}{Q}\right)+o(1)\right).
\end{align*}

Hence~\eqref{eq:ffflow} and~\eqref{eq:p1p2p2low} yield
\begin{align*}
&S_2(\mathcal{A}, x^\theta, x^{1/2}) \\
& \geq \sum_{(r_1, r_2, r_3) \in \mathcal{R}} \frac{x \log^3(1+\Delta)}{\frac{r_1}{2K_1} \frac{r_2}{2K_2} \frac{r_3}{2K_3} \log^3 x} \frac{N^2}{KQ} \left(\mathcal{F}^\sharp\left(\frac{1}{\kappa\left(\frac{r_1-1}{2K_1}\right)}, \frac{1}{\kappa\left(\frac{r_2-1}{2K_2}\right)}, \frac{1}{\kappa\left(\frac{r_3-1}{2K_3}\right)}\right) +o_{Q \to \infty}(1)\right) \\
&= \frac{x^{\alpha}}{\log x} \frac{1}{K_1 \cdot K_2} \sum_{(r_1, r_2, r_3) \in \mathcal{R}} \frac{\mathcal{F}^\sharp\left(\frac{1}{\kappa\left(\frac{r_1-1}{2K_1}\right)}, \frac{1}{\kappa\left(\frac{r_2-1}{2K_2}\right)}, \frac{1}{\kappa\left(\frac{r_3-1}{2K_3}\right)}\right) + o_{Q \to \infty}(1)}{\frac{r_1}{2K_1} \frac{r_2}{2K_2} (1-\frac{r_1}{2K_1}-\frac{r_2}{2K_2})}\\
&= \frac{x^{\alpha}}{\log x} \left( \int_{(\theta-1/2)/3}^{1/3} \int_{\max\{\alpha_2, \theta-3\alpha_2\}}^{\min\{1/2, 1-2\alpha_2\}} \frac{\mathcal{F}^\sharp\left(\frac{1}{\kappa(\alpha_1)}, \frac{1}{\kappa(\alpha_2)}, \frac{1}{\kappa(1-\alpha_1-\alpha_2)}\right)}{\alpha_1 \alpha_2 (1-\alpha_1 - \alpha_2)} \d\alpha_1 \d\alpha_2 + o_{Q \to \infty}(1)\right),
\end{align*}
with the last sum coming from the definition of the Riemann integral.
 Theorem~\ref{thm:shorts-special} now follows.

\section{Sieve with logarithmic weights}
We quickly deduce from Lemma~\ref{le:lowerlinear} the following consequence for sums with logarithmic weights.

\begin{lemma} \label{le:lowerlinearLog}
Let $\varepsilon > 0$ and $Y \geq 1$. Let $\alpha\in (1/2,1)$, and let $\theta \colon [\alpha, 1] \to (1/2, 1)$ be a fixed increasing function. Let $h \colon \mathbb{N} \to \mathbb{R}_{\geq 0}$ be a multiplicative function, and  let $x\geq 2$.

Let $(a_n)_{n \leq x}$ be a sequence of nonnegative real numbers such that, for any $y = x^{\rho} \in [x^\alpha,x]$, the sequence $(a_n)_{n \leq y}$ has level of distribution $\theta(\rho)$ with dimension $1$, size $y/Y$ and density $h$. Then
\begin{align*}
\sum_{x^\alpha < p \leq x} \frac{a_p}{p} &\geq - \frac{e^{\gamma} \log x}{Y} \int_\alpha^1 \int_{\theta(\rho)/2}^{1/2} \prod_{p < x^{\rho (\theta(\rho)-\beta)/2}} \left(1-\frac{h(p)}{p}\right)\frac{\d\beta}{\beta} \d\rho \\
& \qquad + \sum_{r \text{\emph{ even}}} \widetilde{S}_r(\mathcal{A}, \theta) + o\left(\frac{\log x}{Y} \prod_{p < x^{\alpha}} \left(1-\frac{h(p)}{p}\right)\right),
\end{align*}
where 
\[
\widetilde{S}_r(\mathcal{A}, \theta) \coloneqq   \sum_{\substack{x^\alpha < k \leq x \\ k = p_1 \dotsm p_r s, \, p \mid s \implies p > p_r \\ p_r < \dotsb < p_1 < k^{1/2} \\ p_1 \dotsm p_m p_m^2 < k^{\theta\left(\frac{\log k}{\log x}\right)} \text{ for $m < r, m \equiv r \Mod{2}$} \\ p_1 \dotsm p_r p_r^2 \geq k^{(1+\varepsilon)\theta\left(\frac{\log k}{\log x}\right)}}} \frac{a_k}{k}.
\]
\end{lemma}

\begin{proof}
For $t \in [x^\alpha, x]$, let $\mathcal{A}^{(t)} = (a_n)_{n \leq t}$. By partial summation
\begin{align*}
\sum_{x^{\alpha}< p\leq x}\frac{a_p}{p} &= \frac{1}{x} \sum_{x^{\alpha}< p\leq x} a_p + \int_{x^\alpha}^x \frac{1}{t^2} \sum_{x^\alpha < p \leq t} a_p \, \d t \\
& = \frac{1}{x} \left(S(\mathcal{A}^{(x)}, x^{1/2}) + O\left(\sum_{p \leq x^\alpha} a_p\right)\right) + \int_{x^\alpha}^x \frac{1}{t^2}  \left(S(\mathcal{A}^{(t)}, t^{1/2}) + O\left(\sum_{p \leq x^\alpha} a_p\right)\right)\d t.
\end{align*}
Here, by the linear sieve with level of distribution $\theta(\alpha)$,
\[
\sum_{p \leq x^\alpha} a_p \ll \frac{x^\alpha}{Y} \prod_{p < x^\alpha} \left(1-\frac{h(p)}{p}\right),
\]
so the $O$-terms yield acceptable errors.

Applying Lemma~\ref{le:lowerlinear} and making the change of variables $t = x^\rho$, we see that it suffices to show that, for each even $r$, we have
\[
\frac{1}{x} S_r(\mathcal{A}^{(x)}, x^{\theta(1)}, x^{1/2}) + \int_{x^\alpha}^x \frac{1}{t^2} S_r(\mathcal{A}^{(t)}, t^{\theta(\frac{\log t}{\log x})}, t^{1/2}) \d t \geq \widetilde{S}_r(\mathcal{A}, \theta) + o\left(\frac{\log x}{Y} \prod_{p < x^{\alpha}} \left(1-\frac{h(p)}{p}\right)\right).
\]
Here for any $t \in [x^\alpha, x]$
\begin{align*}
 S_r(\mathcal{A}^{(t)}, t^{\theta(\frac{\log t}{\log x})}, t^{1/2}) &= \sum_{\substack{k \leq t \\ k = p_1 \dotsm p_r s, \, p \mid s \implies p > p_r \\ p_r < \dotsb < p_1 < t^{1/2} \\ p_1 \dotsm p_m p_m^2 < t^{\theta\left(\frac{\log t}{\log x}\right)} \text{ for $m < r, m \equiv r \Mod{2}$} \\ p_1 \dotsm p_r p_r^2 \geq t^{\theta\left(\frac{\log t}{\log x}\right)}}} a_k.
\end{align*}
To obtain a lower bound, we can replace the conditions $p_1 < t^{1/2}$ and $p_1 \dotsm p_m p_m^2 \leq t^{\theta\left(\frac{\log t}{\log x}\right)}$ with the same conditions with $t$ replaced by $k$. Furthermore, $p_1 \dotsm p_r p_r^2 \geq t^{\theta\left(\frac{\log t}{\log x}\right)}$ may be replaced by $p_1 \dotsm p_r p_r^2 \geq k^{(1+\varepsilon)\theta\left(\frac{\log k}{\log x}\right)}$ unless $k \leq t^{1-\varepsilon/2}$. Hence, for any $t \in [x^\theta, x]$,
\begin{align*} 
 S_r(\mathcal{A}^{(t)}, t^{\theta(\frac{\log t}{\log x})}, t^{1/2}) & \geq \sum_{\substack{k \leq t \\ k = p_1 \dotsm p_r s, p \mid s \implies p > p_r \\ p_r < \dotsb < p_1 < k^{1/2} \\ p_1 \dotsm p_m p_m^2 < k^{\theta\left(\frac{\log t}{\log x}\right)} \text{ for $m < r, m \equiv r \Mod{2}$} \\ p_1 \dotsm p_r p_r^2 \geq k^{(1+\varepsilon)\theta\left(\frac{\log k}{\log x}\right)}}} a_k + O\left( \sum_{k \leq t^{1-\varepsilon/2}} \tau(k) a_k\right),
\end{align*}
and the claim follows from reversed partial summation since by the level of distribution assumption (with $d = 1$) we have
\[
\sum_{k \leq t^{1-\varepsilon/2}} \tau(k) a_k \ll t^{\varepsilon/4} \sum_{k \leq t^{1- \varepsilon/2}} a_k \ll \frac{t^{1-\varepsilon/4}}{Y}.
\]
This error can be absorbed to the error term in the claim by assumption~\eqref{eq:h(p)condition}.
\end{proof}

\section{A triple convolution estimate in abelian groups}\label{sec: convolution}

Our main result in this section is the following lower bound for triple convolutions that goes into the proof of Theorem \ref{thm:non-torsion-special}. We will give a more general asymmetric version in our forthcoming work.

\begin{proposition} 
\label{prop:Fouriersym}
Let $G$ be an abelian group. Let $\delta \in [0.48, 1/2]$ and $\eta \in [0,\delta]$. Let
$g\colon  G\to [0,1]$ be a function with $\mathbb{E}_{n\in G}g(n) \geq \delta$. Let $a\in G$ and assume that for every subgroup $H \leq G$ of index $2$ we have
\begin{align}
\label{eq:aHlowersym}
\mathbb{E}_{n \in G} g(n)\mathbf{1}_{n \in aH} \geq \frac{1}{2} \eta
\end{align}
Then
\begin{align*}
\frac{1}{|G|^2}(g*g*g)(a)\geq \delta^3 \mathcal{F}(\delta, \eta),
\end{align*}
with $\mathcal{F}(\delta, \eta)$ as in~\eqref{eq:calSpecFdef}.
\end{proposition}

We shall prove the proposition through Fourier analysis, and for this reason we first collect some lemmas concerning the Fourier transform $\widehat{g}(\chi)$. 

\begin{lemma}
\label{le:reallowerbound}
Let $G$ be an abelian group. Let $\delta \in (0, 1]$ and $\eta \in [0,\delta]$. Let
$g\colon  G\to [0,1]$ be a function with $\mathbb{E}_{n\in G}g(n)= \delta$. Let $a\in G$ and let $H$ be a subgroup of index $2$. Assume that
\begin{align*}
\mathbb{E}_{n \in G} g(n)\mathbf{1}_{n \in aH} \geq \frac{1}{2} \eta
\end{align*}
Let $\psi$ be the quadratic character for which $\psi(h) = 1$ for every $h \in H$ (that is, $H=\ker \psi$). Then
\begin{align}
\label{eq:quadcharbounds}
-\delta + \eta \leq \psi(a) \widehat{g}(\psi) \leq \delta.
\end{align}
\end{lemma}

\begin{proof}
The upper bound in~\eqref{eq:quadcharbounds} is trivial. For the lower bound note that
\begin{align*}
\begin{split}
 \psi(a)\widehat{g}(\psi)&=\psi(a)\mathbb{E}_{n\in G}g(n)\psi(n) =\mathbb{E}_{n\in G}g(n)1_{\psi(n)=\psi(a)}-\mathbb{E}_{n\in G}g(n)1_{\psi(n)=-\psi(a)}\\
 &=2\mathbb{E}_{n\in G}g(n)1_{\psi(n)=\psi(a)}-\mathbb{E}_{n\in G}g(n) \geq \eta-\delta.
 \end{split}
\end{align*}
\end{proof}

The following lemma gives a general upper bound for the Fourier transform of a real function at complex characters. 

\begin{lemma}[Pointwise bound for Fourier transform at complex characters] \label{le:complexbound} Let $G$ be an abelian group. Let $\delta\in [0, 1]$, and let $g\colon G\to [0,1]$ be a function such that $\mathbb{E}_{n\in G}g(n)=\delta$. Let $\chi_0$ be the trivial character. Then
\begin{align*}
\max_{\substack{\chi\in \widehat{G}\\\chi^2\neq \chi_0}}|\widehat{g}(\chi)|\leq \max_{\substack{m \geq 3 \\ m \mid |G|}} S(\delta, m),
\end{align*}
where $S(\delta, m)$ is as in~\eqref{eq:Sdelmdef}.
\end{lemma}

\begin{proof}
Let $\chi\in \widehat{G}$ be a complex character. Let $m\geq 3$ be the order of $\chi$. We have 
\begin{align*}
\widehat{g}(\chi)=\sum_{b\Mod m}e\left(\frac{b}{m}\right)w_b,
\end{align*}
where
\begin{align*}
 w_b=\mathbb{E}_{n\in G}g(n)1_{\chi(n)=e(b/m)}  \end{align*}
Note that by the orthogonality of additive characters we have
\begin{align*}
0\leq w_b\leq \mathbb{E}_{n\in G}1_{\chi(n)=e(b/m)} =\frac{1}{m}\sum_{j=0}^{m-1}\mathbb{E}_{n\in G}\chi(n)^{j}e\left(-\frac{jb}{m}\right)=\frac{1}{m}.  
\end{align*}
Hence the claim follows from Lemma~\ref{le:LinCombRootsOfUnity}
\end{proof}

We can figure out which $m$ gives maximal $S(\delta, m)$ when $\delta$ is close to $1/2$. In particular we have the following lemma. 

\begin{lemma} \label{le:Smm=4} Let $\delta\in [0.48, 0.5)$, and let $S(\delta, m)$ be as in~\eqref{eq:Sdelmdef}. Then
\begin{align*}
\max_{m \geq 3} S(\delta, m) = S(\delta, 4) = \frac{1}{4}\sqrt{1+(4\delta-1)^2}. 
\end{align*}
\end{lemma}

\begin{proof}
Recall
\begin{equation}
\label{eq:Sdelmdef2}
S(\delta, m) = \left|\frac{\sin\left(\pi \frac{\left\lfloor \delta m \right\rfloor}{m}\right)}{m \sin\left(\frac{\pi}{m}\right)} + \left(\delta-\frac{\lfloor \delta m\rfloor}{m}\right) e\left(\frac{\lfloor \delta m \rfloor+1}{2m} \right)\right|.
\end{equation}
For $\delta \in [0.48, 0.5)$, we have
\[
\frac{\lfloor \delta m \rfloor+1}{2m} \geq \frac{0.48m - 1 + 1}{2m} = 0.24, \quad \frac{\lfloor \delta m \rfloor+1}{2m} \leq \frac{\frac{m-1}{2}+1}{2m} \leq \frac{1}{4}+\frac{1}{4m},
\]
so that
\[
\frac{1}{2\pi}\operatorname{arg} \left(e\left(\frac{\lfloor \delta m \rfloor+1}{2m} \right)\right) \in \left[0.24, \frac{1}{4}+\frac{1}{4m}\right]. 
\]
Using also that $\delta-\frac{\lfloor \delta m\rfloor}{m} \leq \frac{1}{m}$ and that, for $\delta \in [0.48, 0.5)$, the first term inside the absolute values in~\eqref{eq:Sdelmdef2} is real, positive, and increasing in $\lfloor \delta m\rfloor$ which is at most $\frac{m-1}{2}$, we obtain
\[
S(\delta, m) \leq \left|\frac{\sin\left(\frac{\pi}{2}-\frac{\pi}{2m}\right)}{m \sin\left(\frac{\pi}{m}\right)} + \frac{1}{m}e(0.24) \right|
\]
Here
\[
\frac{\sin\left(\frac{\pi}{2}-\frac{\pi}{2m}\right)}{\sin\left(\frac{\pi}{m}\right)} = \frac{\cos\left(\frac{\pi}{2m}\right)}{\sin\left(\frac{\pi}{m}\right)} = \frac{1}{2\sin\left(\frac{\pi}{2m}\right)}
\]
and thus
\begin{align*}
S(\delta, m)^2 &\leq \left(\frac{1}{2m \sin\left(\frac{\pi}{2m}\right)} + \frac{\cos(0.48\pi)}{m}\right)^2 + \left(\frac{\sin(0.48\pi)}{m}\right)^2 \\
&= \frac{1+4\cos(0.48 \pi) \sin\left(\frac{\pi}{2m}\right) + 4 \sin^2\left(\frac{\pi}{2m}\right)}{4m^2 \sin^2 \left(\frac{\pi}{2m}\right)} \eqqcolon S^+(m),
\end{align*}
say.

Now it is not difficult to see that the numerator of $S^+(m)$ is decreasing and the denominator of $S^+(m)$ increasing as a function of $m$. A quick calculation reveals that $S^+(11) \leq S(0.48, 4)^2 \leq S(\delta, 4)^2$ for any $\delta \in [0.48, 0.5)$. Hence also $S(\delta, m) \leq S(\delta, 4)$ for any $\delta \in [0.48, 0.5)$ and any $m \geq 11$.

For the remaining $m \leq 10$, we notice that, for $\delta \in [0.48, 0.5)$, we have 
\[
\lfloor \delta m\rfloor = 
\begin{cases}
\frac{m}{2}-1 & \text{if $m$ is even;} \\
\frac{m}{2}-\frac{1}{2} & \text{if $m$ is odd,}
\end{cases}
\]
and plugging this into~\eqref{eq:Sdelmdef2}, it is easy to check that $S(\delta, m) \leq S(0.48, 4)$ for every $m \in \{3, 5, 6, 7, 8, 9, 10\}$ and $\delta \in [0.48, 0.5)$.
\end{proof}

\begin{proof}[Proof of Proposition~\ref{prop:Fouriersym}]
Clearly we can assume that $\mathbb{E}_{b \in G} g(b) = \delta$. By the orthogonality of characters, we have
\begin{align*}
\frac{1}{|G|^2}(g*g*g)(a)&= \sum_{\chi\in \widehat{G}}\overline{\chi}(a)\widehat{g}(\chi)^3. 
\end{align*}
Separating the contribution of the trivial character $\chi_0$ and dropping positive terms, we obtain
\begin{align}\label{eq:max}
\frac{1}{|G|^2} (g*g*g)(a)&\geq \delta^3 - \sum_{\substack{\chi\in \widehat{G}\\\chi^2\neq \chi_0}}|\widehat{g}(\chi)|^3+\sum_{\substack{\chi\in \widehat{G}\\\chi^2=\chi_0\\\chi(a)\widehat{g}(\chi)<0}}(\chi(a)\widehat{g}(\chi))^3\nonumber\\
& \geq \delta^3 - \max\left\{\max_{\substack{\chi\in \widehat{G}\\ \chi^2\neq \chi_0}}|\widehat{g}(\chi)|,\max_{\substack{\chi\in \widehat{G}\\ \chi^2=\chi_0,\, \chi(a)\widehat{g}(\chi)<0}}-\chi(a)\widehat{g}(\chi)\right\}\sum_{\substack{\chi \in \widehat{G} \\ \chi \neq \chi_0}} |\widehat{g}(\chi)|^2.
\end{align}
Applying Lemmas~\ref{le:reallowerbound},~\ref{le:complexbound}, and~\ref{le:Smm=4} we see that~\eqref{eq:max} is
\[
\geq \delta^3 - \max\left\{\frac{1}{4}\sqrt{1+(4\delta-1)^2}, \delta - \eta\right\} \sum_{\substack{\chi \in \widehat{G} \\ \chi \neq \chi_0}} |\widehat{g}(\chi)|^2
\]
and we get by inserting $|\widehat{g}(\chi_0)|^2= \delta^2$ the bound
\begin{align*}
\frac{1}{|G|^2} (g*g*g)(a) \geq \delta^3 - \max\left\{\frac{1}{4}\sqrt{1+(4\delta-1)^2}, \delta - \eta\right\}\bigg( -\delta^2 + \sum_{\substack{\chi \in \widehat{G} }} |\widehat{g}(\chi)|^2 \bigg).
\end{align*}
By Parseval's identity  and using $g(b) \in [0,1]$, we see that
\[
\sum_{\substack{\chi \in \widehat{G} }} |\widehat{g}(\chi)|^2 = \E_{b \in G} |g(b)|^2 \leq \E_{b \in G} g(b) = \delta,
\]
and the claim follows.
\end{proof}

\begin{remark}
\label{rem:SecondLargest}
In some ranges of the parameters, it would be possible to improve on Proposition~\ref{prop:Fouriersym} by separating also the contribution of the largest Fourier coefficient in case it corresponds to a quadratic character and then estimating the size of the second largest Fourier coefficient (using the fact that if $|\widehat{g}(\chi)|$ is close to $\delta$ for some real character $\chi$, then $|\widehat{g}(\chi')|$ must be small for all other nontrivial character $\chi'$ by the orthogonality of different characters). This would improve on the lower bound we obtain in Theorem~\ref{thm:non-torsion-special} and thus our value for Linnik's constant.
\end{remark}

\section{Proof of Theorem~\ref{thm:non-torsion-special}} \label{sec:non-torsion-specialproof}

\begin{proof}[Proof of Theorem~\ref{thm:non-torsion-special}]
We apply Lemma~\ref{le:lowerlinearLog} with $a_n = \mathbf{1}_{n \in \mathcal{N}} \mathbf{1}_{f(n) = a}$, $h(n) \mathbf{1}_{n \in \mathcal{N}}$ in place of $h(n)$, $\alpha = 1/2$, and $\varepsilon$ so small that $\theta(1) < 1-\varepsilon$. We drop all of the terms $\widetilde{S}_r$ with $r > 2$. It suffices then to show that
\begin{align}
\label{eq:S2tildeclaimSpec}
\widetilde{S}_2(\mathcal{A}, \theta) \geq \frac{1}{|G|}\left(\frac{\log^3 2}{3} \mathcal{F}\left(\frac{1}{\kappa}, \frac{\eta}{\kappa}\right)+o(1)\right).
\end{align}
Now by definition
\begin{align*}
\begin{aligned}
\widetilde{S}_2(\mathcal{A}, \theta) = \sum_{\substack{x^{1/2} < k \leq x, k \in \mathcal{N} \\ k = p_1 p_2 s, \, p \mid s \implies p > p_2 \\ p_2 < p_1 < k^{1/2} \\ p_1 p_2^3 \geq k^{(1+\varepsilon)\theta\left(\frac{\log k}{\log x}\right)}}} \frac{\mathbf{1}_{f(k) = a}}{k}.
\end{aligned}
\end{align*}
Since we need a lower bound, we can restrict to the case when $s$ is a prime (writing $s = p_3$) and strengthen the lower bound condition for $p_1 p_2^3$ to $p_1 p_2^3 \geq k$ as $(1+\varepsilon)\theta(\frac{\log k}{\log x}) < (1+\varepsilon)(1-\varepsilon) < 1$. We obtain 
\begin{align*}
\begin{aligned}
\widetilde{S}_2(\mathcal{A}, \theta) &\geq \sum_{\substack{x^{1/2} < p_1 p_2 p_3 \leq x \\ p_j \in \mathcal{N}, \, p_3 > p_2 \\ p_2 < p_1 < (p_1 p_2 p_3)^{1/2} \\ p_1 p_2^3 \geq p_1 p_2 p_3}} \frac{\mathbf{1}_{f(p_1 p_2 p_3) = a}}{p_1 p_2 p_3} = \sum_{\substack{x^{1/2} < p_1 p_2 p_3 \leq x \\ p_j \in \mathcal{N}, \, p_3 > p_2 \\ p_2 < p_1 < p_2 p_3 \\ p_2^2 \geq p_3}} \frac{\mathbf{1}_{f(p_1 p_2 p_3) = a}}{p_1 p_2 p_3} \\
&\geq \sum_{\substack{x^{1/6} < p_1, p_2, p_3 \leq x^{1/3} \\ p_j \in \mathcal{N} \\ p_2 < \min\{p_1, p_3\}}} \frac{\mathbf{1}_{f(p_1 p_2 p_3) = a}}{p_1 p_2 p_3}
\end{aligned}
\end{align*}
Using the assumption (A2) we see that
 \begin{align*}
&\sum_{\substack{x^{1/6} < p_1, p_2, p_3 \leq x^{1/3} \\ p_j \in \mathcal{N} \\ p_j \text{ not all distinct}}} \frac{\mathbf{1}_{f(p_1 p_2 p_3) = a}}{p_1 p_2 p_3} \ll  \sum_{x^{1/6} < p,p' \leq x^{1/3}} \frac{\mathbf{1}_{f(p^2 p') = a}}{p^2 p'} \\
 & \ll   x^{-1/6}\sum_{\substack{b,b' \in G \\ b^2 b' = a}}  \sum_{x^{1/6} < p \leq x^{1/3}} \frac{\mathbf{1}_{f(p) = b}}{p}  \sum_{x^{1/6} < p' \leq x^{1/3}} \frac{\mathbf{1}_{f(p') = b'}}{p} \ll  |G|^{-1}x^{-1/6}.
 \end{align*}

Hence, taking into account the number of permutations, we obtain 
\begin{align}
\label{eq:S2midlow}
\widetilde{S}_2(\mathcal{A}, \theta) &\geq \frac{1}{3} \sum_{\substack{x^{1/6} < p_1, p_2, p_3 \leq x^{1/3} \\ p_j \in \mathcal{N}}} \frac{\mathbf{1}_{f(p_1 p_2 p_3) = a}}{p_1 p_2 p_3} + O( |G|^{-1} x^{-1/6}).
\end{align}
For $b \in G$, we define
\[
g(b) =\frac{|G|}{\log 2} \sum_{\substack{x^{1/6} < p \leq x^{1/3} \\ p \in \mathcal{N}}} \frac{\mathbf{1}_{f(p) = b}}{p}
\]
so that
\begin{equation}
\label{eq:S2tillow}
\widetilde{S}_2(\mathcal{A}, \theta) \geq \frac{(\log 2)^3}{3|G|^3} \sum_{\substack{b_1, b_2, b_3 \in G \\ b_1 b_2 b_3 = a}} g(b_1) g(b_2) g(b_3) + O(|G|^{-1} x^{-1/6}).
\end{equation}
Applying assumption (A2), we obtain that, for each $b \in G$ and $j \in \{1, 2, 3\}$, we have $g(b) \leq \kappa$ and by (A3) we have $\mathbb{E}_{b \in G} g(b) = 1+o(1)$. Furthermore assumption (A4) implies that, for any index 2 subgroup $H$, we have 
\begin{equation} \label{eq:QRassumptionforg}
\mathbb{E}_{b\in G}g(b)1_{b \in aH} \geq \frac{1}{2} \eta.
\end{equation}
Given these conditions, normalizing $g$ by $1/\kappa$ and applying Proposition \ref{prop:Fouriersym} we obtain
\begin{align*}
\sum_{\substack{b_1, b_2, b_3 \in G\\ b_1 b_2 b_3 = a}} g(b_1) g(b_2) g(b_3)&\geq (1+o(1)) \mathcal{F}\left(\frac{1}{\kappa}, \frac{\eta}{\kappa}\right)|G|^2.
\end{align*}
Combining this with~\eqref{eq:S2tillow}, the claim~\eqref{eq:S2tildeclaimSpec} follows.
\end{proof}

\begin{remark}
\label{rem:UseMoreSn}
It would be possible to obtain a better lower bound for the left-hand side of~\eqref{eq:FinalLowSpec} by not throwing out so much contribution from the right-hand side of Lemma~\ref{le:lowerlinearLog}.
\end{remark}

\section{Linnik's problem} \label{sec:Linnik}
Before turning to the proof of Theorem~\ref{thm_Linnik} we prove lemmas which will imply the assumptions of Theorem~\ref{thm:non-torsion-special}. We write 
\[
\mathcal{N}_q = \{n \in \mathbb{N} \colon (n, q) = 1\}. 
\]
The first lemma will give us hypothesis Theorem~\ref{thm:non-torsion-special}(A1).
\begin{lemma}\label{le:lod}
Let $\varepsilon > 0$ be small. Let $q\in \mathbb{N}, b \in \mathbb{Z}_q^{\times}$, and $x\geq q^{1+2\varepsilon}$. Let $h \colon \mathcal{N}_q \to \mathbb{R}_{\geq 0}$ be a multiplicative function for which $h(p) = 1$ for every prime $p \in \mathcal{N}_q$. Then the sequence $(1_{n\equiv  b \Mod q})_{n\in \mathcal{N}_q(x)}$ has level of distribution $1-\log q/\log x-\varepsilon$ with dimension $1$, size $x/q$ and density $h$.
\end{lemma}

\begin{proof}
We have
\begin{align*}
\sum_{d \in \mathcal{N}_q(x^{1-\varepsilon}/q)} \left|\sum_{\substack{n\in \mathcal{N}_q(x) \\ n\equiv 0\Mod d}}1_{n\equiv b\Mod q}- \frac{x}{qd}\right| \ll \sum_{d \in \mathcal{N}(x^{1-\varepsilon}/q)} 1 \leq \frac{x^{1-\varepsilon}}{q},
\end{align*}
and the claim follows.
\end{proof}

\begin{remark}
\label{rem:Burgess}
It would be possible to improve on the level of distribution and thus $\theta$ and $\kappa$ for which we apply Theorem~\ref{thm:non-torsion-special} by using bilinear error term in the linear sieve, orthogonality of characters, Parseval, and Burgess' bound for character sums. This would naturally yield an improvement on our value of Linnik's constant.
\end{remark}

The second lemma will give us Theorem~\ref{thm:non-torsion-special}(A2).
\begin{lemma}
\label{le:B-T} 
Let $\varepsilon>0$ be small but fixed. Let $q \in \mathbb{N}$ and $b \in \mathbb{Z}_q^\times$. For any $x\geq q^{1+2\varepsilon}$, we have
\begin{align*}
\sum_{p\leq x} 1_{p\equiv  b \Mod q}\leq \left(\frac{2}{1-\log q/\log x-\varepsilon}\right)\frac{x}{\varphi(q) \log x}.
\end{align*} 
\end{lemma}

\begin{proof}
This is the Brun--Titchmarsh inequality which follows from Lemma~\ref{le:lod} and linear sieve (Lemma~\ref{le:linear}) (adjusting $\varepsilon$).  
\end{proof}

Next we will state two propositions which together show that Theorem~\ref{thm_Linnik} holds in case assumption (A4) of Theorem~\ref{thm:non-torsion-special} fails. Before stating them, we need some notation.
For $L_2\geq L_1\geq 1$, define
\[
\eta^-(L_1, L_2) \coloneqq   2\left(1-\frac{L_2}{L_1+L_2-1}\right) \frac{\log\left(\frac{L_2-1}{L_1}\right)}{\log\frac{L_2}{L_1}},
\]
and
\[
\eta^+(L_1, L_2) \coloneqq   \frac{f_2\left(\frac{2L_2-3}{L_1}\right)}{F_2\left(\frac{L_2-3}{L_1}\right)} \cdot \frac{1}{\log\frac{L_2}{L_1}},
\]  
where $f_2$ and $F_2$ are sieve functions for the two-dimensional sieve --- they are taken to be such that
\begin{align*}
\begin{aligned}
\frac{\d}{\d s} (s^2 F_2(s)) &= 2sf_2(s-1) \\
\frac{\d}{\d s} (s^2 f_2(s)) &= 2sF_2(s-1),
\end{aligned}
\end{align*}
and $s^2 F_2(s) = 21.75$ for $s \in [3.84, 5.84]$ and $f_2(s) = 0$ for $s \leq 4.84$ (see~\cite[Theorem 11.12 and the table in Section 11.19]{Opera}).

\begin{proposition}[Few $p$ with $\psi(p) = -1$] \label{prop:exceptional2}
Let $\varepsilon > 0$ and $L_2 \geq L_1 \geq 1$ be fixed and such that
\begin{equation}
\label{eq:L1L2cond-}
L_1 + 1 +\varepsilon \leq L_2 \leq 3L_1 + 1.
\end{equation}
Assume that, for some quadratic character $\psi \Mod{q}$,
\begin{align}\label{eq:eta-assumption}
\sum_{q^{L_1} < p \leq q^{L_2}} \frac{1-\psi(p)}{p} \leq \eta^-(L_1, L_2) \log\frac{L_2}{L_1} - \varepsilon.
\end{align}
Then, for any $a \in \mathbb{Z}_q^\times$ with $\psi(a) = -1$, we have
\[
\sum_{\substack{p \leq q^{L_1+L_2} \\ p \equiv a \Mod{q}}} 1 \gg \frac{q^{L_1+L_2}}{\varphi(q) \log q},
\]
where the implied constant is effectively computable.
\end{proposition}

\begin{proposition}[Few $p$ with $\psi(p) = 1$] \label{prop:exceptional1}
Let $L_2 \geq L_1 \geq 1$ be fixed. Assume that for some quadratic character $\psi \Mod{q}$ and some fixed $\varepsilon > 0$,
\begin{align}\label{eq:eta+assumption}
\sum_{q^{L_1} < p \leq q^{L_2}} \frac{1+\psi(p)}{p} \leq \eta^+(L_1, L_2) \log\frac{L_2}{L_1} - \varepsilon.
\end{align}
Then, for any $a \in \mathbb{Z}_q^\times$ with $\psi(a) = 1$, we have
\[
\sum_{\substack{p \leq q^{2 L_2} \\ p \equiv a \Mod{q}}} 1 \gg \frac{q^{2L_2}}{\varphi(q)} L(1, \psi) \prod_{\substack{p < q^{L_1} \\ \psi(p) = 1}} \left(1-\frac{2}{p}\right),
\]
where the implied constant is effectively computable.
\end{proposition}

\begin{remark}
\label{rem:eta+-impr}
It would be possible to improve on the results in Proposition~\ref{prop:exceptional2} and Proposition~\ref{prop:exceptional1} by more elaborate arguments. In particular the use of the two-dimensional sieve in proof of Proposition~\ref{prop:exceptional1} is wasteful since the case we are concerned about is that there are lot of primes with $\psi(p) = -1$ which should make the sieve low-dimensional, at least with respect to primes $p \in (q^{L_1},q^{L_2}]$. It would be possible to deduce that one can either improve on the sieve bound or deduce Linnik's theorem for a smaller exponent by studying more carefully the structure of the sieve coefficients.
\end{remark}

\subsection{Proof of Theorem~\ref{thm_Linnik} assuming Propositions \ref{prop:exceptional1} and \ref{prop:exceptional2}}

We shall apply Theorem~\ref{thm:non-torsion-special} with $G = \mathbb{Z}_q^\times$, $a\in G$, $\mathcal{N} = \mathcal{N}_q,$ $f\colon \mathcal{N}_q \to G$ given by $f(n)=n\Mod q$, and $h$ such that $h(n) = 1$ for every $n \in \mathcal{N}_q$. We take $Y = q$ and have $|G| = \varphi(q)$.

Write $L = \Lv$ and let $x = q^B$ with $B \in [L/2, L]$, and let $\varepsilon' > 0$ be small. By Lemma~\ref{le:lod} Theorem~\ref{thm:non-torsion-special}(A1) holds with $\theta = \theta_B \colon [1/2, 1] \to (1/2, 1)$ defined by
\[
\theta_B(\rho) \coloneqq   1-\frac{1}{\rho B} - \varepsilon'
\]
On the other hand, by Lemma~\ref{le:B-T} and partial summation we have, for any $b \in \mathbb{Z}_q^\times$ and any bounded $L_2 > L_1 > 1$,
\begin{align}
\label{eq:kappacalc}
\begin{aligned}
\sum_{\substack{q^{L_1} < p \leq q^{L_2} \\ p \equiv b \Mod{q}}} \frac{1}{p} &\leq \frac{2}{\varphi(q)} \int_{q^{L_1}}^{q^{L_2}} \frac{1}{t \log t(1-\frac{\log q}{\log t}-\varepsilon')} \d t = \frac{2}{\varphi(q)} \int_{L_1}^{L_2} \frac{1}{\beta (1-\frac{1}{\beta})} \d\beta + O\left(\frac{\varepsilon'}{\varphi(q)}\right) \\
& = \frac{2}{\varphi(q)}\left(\log\frac{L_2-1}{L_1-1} + O(\varepsilon') \right).
\end{aligned}
\end{align}
Hence Theorem~\ref{thm:non-torsion-special}(A2) holds with
\[
\kappa = \kappa_B \coloneqq   2 \frac{\log\left(\frac{\frac{B}{3}-1}{\frac{B}{6}-1}\right)}{\log 2} +O(\varepsilon') = 2 + 2\frac{\log\left(\frac{B-3}{B-6}\right)}{\log 2} + O(\varepsilon').
\]
Now if Theorem~\ref{thm:non-torsion-special}(A4) holds with some $\eta$, then, recalling Remark~\ref{rem:SpecSimpleMT}, we obtain
\begin{align*}
\begin{aligned}
\sum_{\substack{q^{B/2} < p \leq q^B \\ p \equiv a \Mod{q}}} \frac{1}{p} &\geq \frac{1}{\varphi(q)}\left(\frac{\log^3 2}{3} \mathcal{F}\left(\frac{1}{\kappa_B},\frac{\eta}{\kappa_B}\right) - \int_{1/2}^1 \int_{\theta_B(\rho)/2}^{1/2} \frac{2}{\theta_B(\rho)-\beta} \frac{\d\beta}{\beta} \frac{\d\rho}{\rho} + o(1) \right) \\
& \eqqcolon \frac{1}{\varphi(q)} \mathcal{C}(\eta, \kappa_B, \theta_B),
\end{aligned}
\end{align*}
say. A computation reveals that, once $\varepsilon'$ is sufficiently small, for every $B \in [L/2, L]$, we have
\[
\mathcal{C}(0.249, \kappa_B, \theta_B) > 0.
\] 
Hence Theorem~\ref{thm:non-torsion-special} implies the claim unless Theorem~\ref{thm:non-torsion-special}(A4) fails with $\eta = 0.249$ for $x = q^B$ for every $B \in [L/2, L]$. Hence we can assume that, for each $B \in [L/2, L]$, there exists a subgroup $H_B \leq \mathbb{Z}_q^\times$ of index $2$ such that
\begin{equation}
\label{eq:aHup}
\sum_{\substack{q^{B/6} < p \leq q^{B/3} \\ p \in aH_B}} \frac{1}{p} <  0.249 \cdot \frac{\log 2 + o(1)}{2}.
\end{equation}

For any $B \in [L/2, L]$ we have $\eta^-(B/6, B/3) \geq 0.62$. Hence the claim holds by Proposition~\ref{prop:exceptional2} if $a \not \in H_B$ for some $B \in [L/2, L]$, so we can assume that $a \in H_B$ for every $B \in [L/2, L]$.

Consider next the case $H_{L/2} = H_L$. Then
\[
\sum_{\substack{q^{B/12} < p \leq q^{B/3} \\ p \in aH_B}} \frac{1}{p} <  0.249 \cdot \frac{\log 4 + o(1)}{2}.
\]
But we have $\eta^+(L/12, L/3) \geq 0.25$ so in this case the claim holds by Proposition~\ref{prop:exceptional1}.

Hence we can assume that $H_{L/2} \neq H_L$. Consequently we can find $B_1, B_2 \in [L/2, L]$ such that 
\[
B_1 < B_2 \leq 1.01 B_1
\] 
and $H_{B_1} \neq H_{B_2}$. Write $I = [1.01B_1/6, B_1/3]$ and notice that $I \subseteq [B_j/6, B_j/3]$ for $j = 1, 2$. Hence by~\eqref{eq:aHup}, we obtain that, for $j = 1, 2$, 
\begin{equation}
\label{eq:aHjup}
\sum_{\substack{p \in q^I \\ p \in H_{B_j}}} \frac{1}{p} < 0.249 \cdot \frac{\log 2 + o(1)}{2} < \frac{0.125 \cdot \log 2+o(1)}{\log(2/1.01)} \cdot \sum_{\substack{p \in q^I}} \frac{1}{p} < (0.13 + o(1)) \cdot \sum_{\substack{p \in q^I}} \frac{1}{p}
\end{equation}
Let 
\[
g(b) \coloneqq   \frac{\varphi(q)}{2 \log\left(\frac{B_1/3-1}{1.01B_1/6-1}\right)+\varepsilon'} \sum_{\substack{p \in q^I \\ p \equiv b \Mod{q}}} \frac{1}{p}
\]
By~\eqref{eq:kappacalc} we see that $g(b) \leq 1$ for every $b \in \mathbb{Z}_q^\times$ and by Mertens' theorem
\[
\mathbb{E}_{b\in \mathbb{Z}_q^\times} g(b) = \frac{\log \frac{2}{1.01}+o(1)}{2 \log\left(\frac{B_1/3-1}{1.01B_1/6-1}\right)+\varepsilon'}. 
\]
Similarly, by~\eqref{eq:aHjup} we also see that for $j \in \{1,2\}$
\begin{align*}
\frac{1}{|G|} \sum_{b \in H_{B_j}} g(b) \leq 0.13 \cdot \frac{\log \frac{2}{1.01}+o(1)}{2 \log\left(\frac{B_1/3-1}{1.01B_1/6-1}\right)+\varepsilon'}.
\end{align*}
Let $c \in G$ such that $c \notin H_{B_1}, H_{B_2}$ and note that $H_{B_1} \cap H_{B_2}$ is a subgroup of index $4$ and
\[
G \setminus (cH_{B_1} \cup cH_{B_2}) =  H_{B_1} \cap H_{B_2}.
\] 
Thus, using $g(b) \leq 1$ we get by the inclusion-exclusion principle
\begin{align*}
\frac{1}{4} &\geq \frac{1}{|G|}\sum_{b \in  c(  H_{B_1} \cap  H_{B_2})} g(b)= \frac{1}{|G|}\sum_{b \in  ( c H_{B_1} \cap c H_{B_2})} g(b)   \\
&= \frac{1}{|G|}\sum_{b \in  G} g(b) -   \frac{1}{|G|}\sum_{b \in   H_{B_1}} g(b) - \frac{1}{|G|}\sum_{b \in   H_{B_2}} g(b)  +  \sum_{b \in H_{B_1} \cap H_{B_2}}g(b)\\
&\geq \frac{1}{|G|} \sum_{b \in G} g(b) - \frac{1}{|G|}\sum_{b \in H_{B_1}} g(b) - \frac{1}{|G|}\sum_{b \in H_{B_2}} g(b) \\
&\geq (1-2 \cdot 0.13) \cdot \frac{\log \frac{2}{1.01}+o(1)}{2 \log\left(\frac{B_1/3-1}{1.01B_1/6-1}\right)+\varepsilon'} \geq (1-2 \cdot 0.13) \cdot 0.4876  \geq 0.36,
\end{align*}
which is a contradiction. \qed

\subsection{Proof of Proposition \ref{prop:exceptional2}}
We first prove Proposition~\ref{prop:exceptional2} which is easier. Let $\varepsilon' > 0$ be sufficiently small in terms of $\varepsilon$. Lemma~\ref{le:lod} implies that, for any $b \in \mathbb{Z}_q^\times$ and any $L' \in [1, L_1 +L_2]$, the sequence $(1_{n \equiv b \pmod{q}})_{n \in \mathcal{N}_q(q^{L'})}$ has level of distribution $(L'-1-\varepsilon')/L'$ with dimension $1$, size $q^{L'-1}$ and density function $h(d)=1$ for $d \in \mathcal{N}_q$. 

Note that, since $\psi(a) = -1$, if $n \equiv a \pmod{q}$, then $n$ has at least one prime factor for which $\psi(p)=-1$. Thus
\begin{align*}
&\sum_{\substack{p \leq q^{L_1+L_2} \\ p \equiv a \pmod{q}}} 1 \geq \sum_{\substack{n \leq q^{L_1+L_2}\\ n \equiv a \pmod{q} \\ (n, P(q^{L_1})) = 1}} 1 - \sum_{\substack{q^{L_1} < p \leq q^{L_2} \\ \psi(p) = -1}} \sum_{\substack{n \leq q^{L_1+L_2}/p \\ n \equiv a\overline{p} \pmod{q} \\ p' \mid n \implies p' \in [q^{L_1}, q^{L_2}]}} 1.
\end{align*}
By~\eqref{eq:L1L2cond-} we have $(L_1+L_2-1-\rho)/3 < L_1$ for every $\rho \in [L_1, L_2]$ and thus we can replace the condition $p' \mid n \implies p' \in [q^{L_1}, q^{L_2}]$ by the weaker condition $(n, P((q^{L_1+L_2-1-\varepsilon'}/p)^{1/3}) = 1$. We then obtain by the linear sieve (Lemma \ref{le:linear} and values of $f_1(s)$ and $F_1(s)$) a lower bound
\begin{align*}
\sum_{\substack{p \leq q^{L_1+L_2} \\ p \equiv a \pmod{q}}} 1 &\geq \frac{q^{L_1+L_2}}{q} \left(f_1\left(\frac{L_1+L_2-1-\varepsilon'}{L_1}\right) + o(1)\right) \prod_{\substack{p < q^{L_1} \\ p \nmid q}} \left(1-\frac{1}{p}\right) \\
& \qquad - \frac{q^{L_1+L_2}}{q} \sum_{\substack{q^{L_1} < p \leq q^{L_2} \\ \psi(p) = -1}} \frac{1}{p} (F_1(3)+o(1)) \prod_{\substack{p' < (q^{L_1+L_2-1-\varepsilon'}/p)^{1/3} \\ p' \nmid q}} \left(1-\frac{1}{p}\right) \\
&\geq 2\frac{q^{L_1+L_2}}{\varphi(q)} \left(\frac{\log\left(\frac{L_1+L_2-1-\varepsilon'}{L_1}-1\right)}{L_1+L_2-1-\varepsilon'}  - \sum_{\substack{q^{L_1} < p \leq q^{L_2} \\ \psi(p) = -1}} \frac{1}{p(L_1+L_2-1-\varepsilon'-\frac{\log p}{\log q})} + o(1)\right) \\
&\geq 2\frac{q^{L_1+L_2}}{\varphi(q)} \left(\frac{\log\left(\frac{L_2-1-\varepsilon'}{L_1}\right)}{L_1+L_2-1-\varepsilon'}  - \frac{\eta^-(L_1, L_2) \log\left(\frac{L_2}{L_1}\right)-\varepsilon}{2(L_1-1-\varepsilon')} + o(1)\right), \\
\end{align*}
and the claim follows once $\varepsilon'$ is sufficiently small. \qed

\subsection{Preparations for the proof of Proposition~\ref{prop:exceptional1}}
The proof of Proposition~\ref{prop:exceptional1} is more involved. We will sieve the sequence $1 \ast \psi$ and thus need a level of distribution result for it. Following~\cite[Section 24.2]{Opera} we establish the following.

\begin{lemma}
\label{le:1astchi_dAP} Let $\varepsilon > 0$. Let $q \in \mathbb{N}$ and let $\psi$ be a quadratic character $\Mod{q}$. Let $a \in \mathbb{Z}_q^\times$, and let $d \in \mathbb{N}$ and $x \in \mathbb{R}_+$ be such that $x \geq \max\{q, d\}$.
Then
\begin{align}
\label{eq:1astchi_d}
\sum_{\substack{n \leq x \\ \d \mid n \\ n \equiv a \Mod{q}}} (1 \ast \psi)(n) = \mathbf{1}_{(d, q) = 1} (1+\psi(a)) h(d) L(1, \psi) \frac{x}{q} + O_\varepsilon\left(\frac{x^{1/2+\varepsilon}}{d^{1/2}}\right),
\end{align}
where $h$ is a multiplicative function such that, for every $p \in \mathbb{P}$,
\begin{align}
\label{eq:hdef}
h(p) = \frac{1+\psi(p)}{p}-\frac{\psi(p)}{p^2} \quad \text{and} \quad |h(p^k)| \leq \frac{2k}{p^k} \, \text{ for any $k \in \mathbb{N}$.}
\end{align}
\end{lemma}

Before turning to the proof of Lemma~\ref{le:1astchi_dAP} we establish two lemmas from which the proof will quickly follow. The first lemma concerns the case $d= 1$ of Lemma~\ref{le:1astchi_dAP}.

\begin{lemma}
\label{le:1astpsiAP}
Let $\varepsilon>0$ be fixed. Let $q \in \mathbb{N}$ and $a \in \mathbb{Z}_q^\times$. Let $\psi$ be a quadratic character $\Mod{q}$ and let $x \geq q$. Then
\[
\sum_{\substack{n \leq x \\ n \equiv a\Mod{q}}} (1 \ast \psi)(n) = (1+\psi(a))L(1, \psi) \frac{x}{q} + O(x^{1/2})
\]
\end{lemma}

\begin{proof}
As in~\cite[Section 24.2]{Opera} we use a standard hyperbola method argument. We write
\[
(1 \ast \psi)(n) = \sum_{k\ell = n} \psi(k) = \sum_{\substack{k \ell = n \\ k < \ell}} \psi(k) + \sum_{\substack{k\ell = n \\ \ell < k}} \psi(n/\ell) + O(\mathbf{1}_{\sqrt{n} \in \mathbb{N}}) = (1+\psi(n))\sum_{\substack{k\ell = n \\ k < \ell}} \psi(k)+ O(\mathbf{1}_{\sqrt{n} \in \mathbb{N}}).
\]
Hence
\begin{align*}
\sum_{\substack{n \leq x \\ n \equiv a\Mod{q}}} (1 \ast \psi)(n)
&= (1+\psi(a)) \sum_{k \leq x^{1/2}} \psi(k) \sum_{\substack{k < \ell \leq x/k \\ \ell \equiv a \overline{k} \pmod{q}}} 1 +O(x^{1/2}) \\
&= (1+\psi(a)) \sum_{k \leq x^{1/2}} \left(\frac{x}{q} \cdot \frac{\psi(k)}{k} - \psi(k) \frac{k}{q} +O(1)\right) + O(x^{1/2}).
\end{align*}
Using partial summation and the trivial bound $|\sum_{n \leq y} \psi(n)| \leq \varphi(q)$ for any $y \geq 1$, 
we obtain 
\begin{align*}
&\sum_{\substack{n \leq x \\ n \equiv a\Mod{q}}} (1 \ast \psi)(n) = (1+\psi(a)) L(1, \psi) + O(x^{1/2}),
\end{align*}
as claimed.
\end{proof}

The next lemma will allow us to reduce establishing Lemma~\ref{le:1astchi_dAP} to establishing Lemma~\ref{le:1astpsiAP}. The needed case $f = 1, g = \psi$ is stated without proof in~\cite[(24.5)]{Opera}. We provide the proof for completeness.

\begin{lemma}
\label{le:fastgdec}
Let $f,g \colon \mathbb{N} \to \mathbb{C}$ be completely multiplicative functions. Then, for any $m, n \in \mathbb{N}$, we have
\[
(f \ast g)(mn) = \sum_{r \mid (m, n)} \mu(r) f(r) g(r) (f \ast g)\left(\frac{m}{r}\right) (f \ast g)\left(\frac{n}{r}\right).
\]
\end{lemma}

\begin{proof}
We have
\[
(f \ast g)(mn) = \sum_{ab = mn} f(a) g(b).
\]
We write $c = (a, m)$ so that $c \mid m$ and $(\frac{a}{c}, \frac{m}{c}) = 1$. Since $ab = mn$ we also must have $\frac{m}{c} \mid b$. Inserting this, and then using M\"obius inversion to detect the condition $(\frac{a}{c}, \frac{m}{c}) = 1$, we obtain
\[
(f \ast g)(mn) = \sum_{c \mid m} \sum_{\substack{ab = mn \\ c \mid a, \frac{m}{c} \mid b \\ (\frac{a}{c}, \frac{m}{c}) = 1}} f(a) g(b) = \sum_{cr \mid m} \mu(r) \sum_{\substack{ab = mn \\ cr \mid a, \frac{m}{c} \mid b}} f(a) g(b).
\]
Next we write $a = cr \cdot k$ and $b = \frac{m}{c} \cdot \ell = r \cdot \frac{m}{cr} \cdot \ell$ (note that since $r \mid \frac{m}{c}$, here $\frac{m}{cr} \in \mathbb{N}$). Now $ab = mn$ corresponds to $k\ell r = n$. Inserting these and using multiplicativity, we obtain 
\begin{align*}
(f \ast g)(mn) &= \sum_{cr \mid m} \mu(r) \sum_{\substack{k \ell r = n}} f(crk) g\left(r \cdot \frac{m}{cr} \cdot \ell\right) \\
&= \sum_{r \mid (m, n)} \mu(r)f(r)g(r) \sum_{c \mid \frac{m}{r}} f(c) g\left(\frac{m}{cr}\right) \sum_{k \ell = \frac{n}{r}} f(k) g\left(\ell\right),
\end{align*}
and the claim follows.
\end{proof}

Now we are ready to prove Lemma~\ref{le:1astchi_dAP}, still following~\cite[Section 24.2]{Opera}.
\begin{proof}[Proof of Lemma~\ref{le:1astchi_dAP}]
We can clearly assume that $(d, q) = 1$ and $\psi(a) = 1$. Using Lemma~\ref{le:fastgdec} we see that
\begin{align*}
\sum_{\substack{n \leq x \\ d \mid n \\ n \equiv a \Mod{q}}} (1 \ast \psi)(n) &= \sum_{\substack{m \leq x/d \\ m \equiv a\overline{d} \Mod{q}}} (1 \ast \psi)(md) \\
&= \sum_{r \mid d} \mu(r) \psi(r) (1 \ast \psi)\left(\frac{d}{r}\right) \sum_{\substack{m \leq x/d \\ r \mid m \\ m \equiv a \overline{d} \Mod{q}}} (1 \ast \psi)\left(\frac{m}{r}\right).
\end{align*}
First writing $m = rn$ and then applying Lemma~\ref{le:1astpsiAP}, we obtain
\begin{align*}
\sum_{\substack{n \leq x \\ d \mid n \\ n \equiv a \Mod{q}}} (1 \ast \psi)(n) &= \sum_{r \mid d} \mu(r) \psi(r) (1 \ast \psi)\left(\frac{d}{r}\right) \sum_{\substack{n \leq x/(dr)\\ n \equiv a \overline{dr} \Mod{q}}} (1 \ast \psi)(m) \\
&= \sum_{r \mid d} \mu(r) \psi(r) (1 \ast \psi)\left(\frac{d}{r}\right) (1+\psi(a \overline{dr})) L(1, \psi) \frac{x}{dr q} + O\left(\frac{x^{1/2+\varepsilon}}{d^{1/2}}\right).
\end{align*}
Now $\psi(a\overline{dr}) = \psi(a) \psi(d/r) = \psi(d/r)$ and it is easy to see that $(1 \ast \psi)(d/r) (1+\psi(d/r)) = 2 (1 \ast \psi)(d/r)$ and so
\begin{align*}
&\sum_{\substack{n \leq x \\ d \mid n \\ n \equiv a \Mod{q}}} (1 \ast \psi)(n) = 2L(1, \psi) \frac{x}{q} \sum_{r \mid d} \frac{\mu(r) \psi(r)}{dr} (1 \ast \psi)\left(\frac{d}{r}\right)  + O\left(\frac{x^{1/2+\varepsilon}}{d^{1/2}}\right).
\end{align*}
The claim follows with 
\[
h(d) = \sum_{r \mid d} \frac{\mu(r) \psi(r)}{dr} (1 \ast \psi)\left(\frac{d}{r}\right),
\]
as it is easy to establish that this satisfies~\eqref{eq:hdef}.
\end{proof}

When we apply Lemma~\ref{le:1astchi_dAP} we need the following classical lower bound for $L(1, \psi)$. 

\begin{lemma}[Lower bound for $L(1,\psi)$]\label{le_MV} Let $q\in \mathbb{N}$ and let $\psi\Mod q$ be a quadratic character. Then we have $L(1,\psi)\gg q^{-1/2}$, with the implied constant being effective. 
\end{lemma}

\begin{proof}
See e.g.~\cite[Theorem 11.11]{MonVau}. 
\end{proof}

\begin{remark}
\label{rem:eta+impr}
Of course using Siegel's ineffective bound $L(1, \chi) \gg_\varepsilon q^{-\varepsilon}$ would yield a small improvement to Proposition~\ref{prop:exceptional1}. Also, it would be possible to improve on the error term in Lemma~\ref{le:1astpsiAP} by being more careful with the error terms (compare with the divisor problem of obtaining asymptotics for $\sum_{n \leq x} \tau(n)$) which would also yield an improvement on Proposition~\ref{prop:exceptional1}.
\end{remark}

\subsection{Proof of Proposition~\ref{prop:exceptional1}}
We are now ready to prove Proposition \ref{prop:exceptional1}. Let $\varepsilon' > 0$ be small in terms of $\varepsilon$. Let $b \in \mathbb{Z}_q^\times$ with $\psi(b) = 1$. By Lemmas~\ref{le:1astchi_dAP} and~\ref{le_MV}, the sequence $((1\ast \psi)(n) \mathbf{1}_{n \equiv b \Mod{q}})_{n \leq q^{L'}}$ has for any $L' \in [1, 2L_2]$ level of distribution $(L'-3-\varepsilon')/L'$ with dimension $2$, size $2L(1, \psi)q^{L'-1}$ and density $h$, with $h$ as in Lemma~\ref{le:1astpsiAP}.

We have by Buchstab's identity
\begin{align*}
2\sum_{\substack{p \leq q^{2L_2} \\ p \equiv a \pmod{q}}}1&\geq \sum_{\substack{p \leq q^{2L_2} \\ p \equiv a \pmod{q}}} (1 \ast \psi)(p) \\
&= \sum_{\substack{n \leq q^{2L_2}\\ n \equiv a \pmod{q} \\ (n, P(q^{L_1})) = 1}} (1\ast \psi)(n) - \sum_{\substack{q^{L_1} \leq p_1 \leq q^{L_2}}} (1 \ast \psi)(p_1) \sum_{\substack{n \leq q^{2L_2}/p_1 \\ n \equiv a\overline{p_1} \pmod{q} \\ p \mid n \implies p \geq p_1}} (1 \ast \psi)(n)\\
&\geq \sum_{\substack{n \leq q^{2L_2}\\ n \equiv a \Mod{q} \\ (n, P(q^{L_1})) = 1}} (1\ast \psi)(n) - \sum_{\substack{q^{L_1} \leq p_1 \leq q^{L_2}}} (1 \ast \psi)(p_1) \sum_{\substack{n \leq q^{2L_2}/p_1 \\ n \equiv a\overline{p_1} \pmod{q} \\ (n,P(q^{L_1}))=1}} (1 \ast \psi)(n).
\end{align*}
Note that in the second term $p_1 \leq q^{L_2}$ so we can there use a sieve of level $D = q^{L_2-3-\varepsilon'}$. By the two-dimensional sieve (see e.g.~\cite[Theorem 11.12]{Opera}), we obtain a lower bound
\begin{align*}
&2\sum_{\substack{p \leq q^{2L_2} \\ p \equiv a \pmod{q}}}1 \geq  \left(f_2\left(\frac{2L_2-3-\varepsilon'}{L_1}\right)+o(1)\right) L(1, \psi) \frac{2q^{2L_2}}{q} \prod_{\substack{p < q^{L_1}}} \left(1-h(p)\right) \\
&\quad -  \sum_{q^{L_1} < p_1 \leq q^{L_2}} \frac{1+\psi(p_1)}{p_1} \frac{2q^{2L_2}}{q} L(1, \psi) \left(F_2\left( \frac{L_2-3-\varepsilon'}{L_1}\right)+o(1)\right) \prod_{\substack{p < q^{L_1}}} \left(1-h(p)\right).
\end{align*}
Hence, by~\eqref{eq:eta+assumption} we have
\begin{align*}
&\sum_{\substack{p \leq q^{2L_2} \\ p \equiv a \pmod{q}}} (1 \ast \psi)(p) \geq   L(1, \psi) \frac{2q^{2L_2}}{q} \prod_{\substack{p < q^{L_1}}} \left(1-h(p)\right) \\
&\qquad \cdot \left(f_2\left(\frac{2L_2-3-\varepsilon'}{L_1}\right) - \left(\eta^+(L_1, L_2)\log\frac{L_2}{L_1} - \varepsilon\right) F_2\left(\frac{L_2-3-\varepsilon'}{L_1}\right)+o(1)\right),
\end{align*}
and the claim follows once $\varepsilon'$ is sufficiently small in terms of $\varepsilon$. \qed

\section*{Acknowledgements}
KM was partially supported by Academy of Finland grant no. 285894. JM was supported by the European Research Council (ERC) under the European Union's Horizon 2020 research and innovation programme (grant agreement No 851318).  JT was supported by a von Neumann Fellowship (NSF grant no. DMS-1926686) and funding from the European Union's Horizon Europe research and innovation programme under Marie Sk\l{}odowska-Curie grant agreement no. 101058904.  The authors would like to thank Andrew Granville for helpful discussions.

\bibliography{refs}
\bibliographystyle{plain}

\end{document}